\documentclass[12pt]{article}

\usepackage{amsfonts}
\usepackage{amsthm}
\usepackage{euscript}
\usepackage{graphicx}
\usepackage{amsfonts}
\usepackage{amsthm}
\usepackage{euscript}
\usepackage{amsmath}
\usepackage{amssymb}
\usepackage{wrapfig}
\usepackage{caption}
\usepackage{indentfirst}
\usepackage{cmap}
\usepackage{makeidx}
\usepackage{mathrsfs}

\usepackage[colorlinks, citecolor=black, urlcolor=blue, linkcolor=black]{hyperref}

\newtheorem{theorem}{Theorem}[section]
\newtheorem{lemma}{Lemma}[section]
\newtheorem{corollary}{Corollary}[section]
\newtheorem*{corollary*}{Corollary}

\newtheorem{remark}{Remark}[section]

\makeatletter
\@addtoreset{equation}{section}
\makeatother

\begin{document}

\thispagestyle{plain} 

\sloppy

\begin{center}
{\bf \Large\bf Eigenvalues asymptotics of unbounded operators. Two-photon quantum Rabi model.}
\\ \bigskip
{\large Ianovich E.A.}
\\ \bigskip
   {\it Saint Petersburg, Russia}
\\ \bigskip
   {\it\small E-mail: eduard@yanovich.spb.ru}
\end{center}

\begin{abstract}
In this work the general results about asymptotics of eigenvalues of unbounded operators are obtained. We consider here different cases of compact, relatively compact, selfadjoint or nonselfadjoint perturbations. In particular we prove a generalization of Janas-Naboko lemma about eigenvalues asymptotics of unbounded operators at compact perturbation. A generalization of our previous result about noncompact perturbation of oscillator spectrum is also given. As an example we consider two-photon quantum Rabi model. We obtain tree-term asymptotic formula for large eigenvalues of the energy operator of this model. The asymptotics of related to this model polynomials is found. We give also an original proof of the Perelomov factorization theorem for contraction operator of quantum optics.
\end{abstract}

\section{Introduction}

In works~\cite{11,12} we proposed a method which gives the posiibility to find one or two first terms of asymptotics of large eigenvalues for a wide class of operators. But it does not allow to find a third term for such interesting physical operators as Jaynes-Cummings model without rotating wave approximation (one-photon quantum Rabi model) for example and many others. In this work we generalize our previous results. Section 2 is dedicated to the different properties of compact operators which are used next. In the section~\ref{main_theorems_about_asymptotics} we prove the general theorems~(\ref{eigenvalues_at_relatively_compact_R}) and~(\ref{noncompact_perturbation_oscillator}) about relatively compact perturbation of arbitrary spectrum and non-compact perturbation of oscillator spectrum. In the section~\ref{Rabi} we consider the two-photon quantum Rabi model as an example of application of the theorem~(\ref{noncompact_perturbation_oscillator}). In particular we prove here an important Perelomov factorization theorem~(\ref{factorization_U}). Section 5 is dedicated to the finding of asymptotics of the polynomials (theorems~(\ref{asymptot_Polynom_even}) and~(\ref{asymptot_Polynom_odd})) related to this  Rabi model. In the section 6 we prove three-term asymptotic formula for eigenvalues of the Hamiltonian of two-photon quantum Rabi model (theorem~(\ref{final_theorem_asyptot})).

\section{Several remarkable properties of compact operators}

Let \(B(H)\) be the ring of bounded operators acting in the Hilbert space \(H\). It is well known~\cite{6,7} that the set of compact (completely continuous) operators \(G_\infty\subset B(H)\) is a two-sided symmetric ideal of \(B(H)\):
\begin{equation}
\label{Ideal_property}
\left(A\in G_\infty\,,\: C\in B(H)\:\right) \Rightarrow \left(AC\in G_\infty\,,\: CA\in G_\infty\right)\,,
\end{equation}
\begin{equation}
\label{Symmetric_property}
A\in G_\infty \Rightarrow A^*\in G_\infty\,.
\end{equation}

Any compact operator \(A\) admits the Schmidt expansion~\cite{6}
\begin{equation}
\label{Schmidt_expansion}
A=\sum_{n=1}^\infty s_n\,(\,\cdot\,,\varphi_n)\,\psi_n\,,
\end{equation}
where \(s_n\) are {\it s}-numbers of \(A\) and \(\{\varphi_n\}\), \(\{\psi_n\}\) are certain orthonormal systems in \(H\). 

s-numbers \(s_n\) are the eigenvalues of selfadjoint operator \(\sqrt{A^*A}\) and have the following properties:
$$
s_1=\|A\|=\|A^*\|\,,\quad\sqrt{A^*A}\,\varphi_n=s_n\,\varphi_n\,,\quad  s_n\ge 0\,,\quad s_{n+1}\le s_n\,,\quad n\in\mathbb{N}\,,\quad \lim\limits_{n\to\infty}s_n=0\,.
$$
Besides \(\psi_n=U\varphi_n\), where \(U\) is an isometric operator and \(A=U\sqrt{A^*A}\).

The compact operator transforms any weakly convergent sequence into strongly convergent one~\cite{9}:
\begin{equation}
\label{sequence_convergence_property}
\left(A\in G_\infty\,,\:\forall f\in H\:\:(x_n,f)\to(x,f)\:\right)\Rightarrow \|Ax_n-Ax\|\to 0\,.
\end{equation}

\begin{lemma}
\label{lemma_compact}
Let \(\{e_n\}\) be an infinite orthonormal system of vectors and \(A\in G_\infty\). Then \(\|A\,e_n\|\to 0\) as \(n\to\infty\).
\end{lemma}

\begin{proof}[1st Proof]
From (\ref{Schmidt_expansion}) it follows that \(A\,e_n=\sum\limits_{m=1}^\infty s_m\,(e_n,\varphi_m)\,\psi_m\) and
$$
t_n\equiv\|A\,e_n\|^2=\sum\limits_{m=1}^\infty s_m^2\,|(e_n,\varphi_m)|^2\,.
$$
The matrix \(C_{n,m}=|(e_n,\varphi_m)|^2\) satisfies the all conditions of regular transformation~\cite{8}:
$$
\sum\limits_{m=1}^\infty |C_{n,m}|=\sum\limits_{m=1}^\infty|(e_n,\varphi_m)|^2\le \|e_n\|^2=1\,,
$$
$$
\lim\limits_{n\to\infty}C_{n,m}=\lim\limits_{n\to\infty}|(e_n,\varphi_m)|^2=0\,,\quad m\in\mathbb{N}\,.
$$
Hence from \(s_m^2\to 0\) it follows that \(t_n\to 0\).
\end{proof}

\begin{proof}[2nd Proof] The sequence \(e_n\) weakly converges to zero: \(\forall f\in H\:\:(e_n,f)\to0\). According to~(\ref{sequence_convergence_property}), the compact operator \(A\) transforms \(e_n\) into strongly convergent sequence \(Ae_n\). Therefore \(\|Ae_n\|\to0\).

\end{proof}

The operator \(A\) can be represented by infinite matrix \(\{a_{ij}\}\) in some orthonormal bases \(\{e_n\}\): \(a_{ij}=(Ae_j,e_i)\). Infinite matrix is called to be band if  \(a_{ij}=0\) as \(|i-j|>r\ge 0\), where \(r\in\mathbb{Z}_+\) is fixed.  For infinite band matrices there exists a simple compactness criterion~\cite{10}:

\begin{theorem}
\label{compactness_for_band_matrix}
The operator \(A\) represented by band matrix \(\{a_{ij}\}\) is compact iff 
$$
\lim\limits_{n\to\infty}a_{n,n+k}=0\,,\quad k=-r,-(r-1),\ldots,(r-1),r\,.
$$
\end{theorem}

For not band matrices there is no such simple criterion. There are many sufficient criteria only.
For example, there exists the following theorem~\cite{6,9,10}:

\begin{theorem}
\label{compactness_Hilbert-Schmidt}
If the matrix \(\{a_{ij}\}\) of the operator \(A\) satisfies the condition
$$
\sum\limits_{i,j=1}^\infty|a_{ij}|^2<\infty\,,
$$
then the operator \(A\) is compact.
\end{theorem}

However this condition is sufficient and necessary for \(A\in G_2\subset G_\infty\), where \(G_2\) is a Hilbert-Schmidt class~\cite{6,10}.

Also it is known the following result~\cite{6,9,10}:
\begin{theorem}
\label{compactness_Nuclear}
If the matrix \(\{a_{ij}\}\) of the operator \(A\) satisfies the conditions
$$
\forall j\in\mathbb{N}\:\:\:\sum_{i=1}^\infty|a_{ij}|\le L\,,\quad \forall i\in\mathbb{N}\:\:\:\sum_{j=1}^\infty|a_{ij}|\le L\,,\quad
\sum_{i=1}^\infty|a_{ii}|<\infty\,,
$$
then the operator \(A\) is compact and nuclear, that is \(A\in G_1\subset G_2\subset G_\infty\).
\end{theorem}

We recall also~\cite{6} that the compact operator \(A\) belongs to the Schatten-Neuman class \(G_p\) (\(1\le p\le\infty\)) iff
$$
\|A\|_p\equiv\left[\sum\limits_{n=1}^\infty s^p_n(A)\right]^{1/p}<\infty\,.
$$

We will need further the following sufficient compactness condition for infinite not band matrices proved in~\cite{11}:

\begin{theorem}
\label{compactness_for_nonband_matrix}
Let V be a bounded not compact operator in Hilbert space \(H\) represented in some orthonormal basis by infinite matrix \(\{V_{ij}\}_{i,j=0}^\infty\) satisfying the condition
\begin{equation}
\label{condition_1}
\lim\limits_{n\to\infty}V_{n,n+p}=0\,,\quad p\in\mathbb{Z}\,.
\end{equation}
Let \(\{b_i\}_{i=-\infty}^\infty\) be an arbitrary \(l_2\)-sequence:
\begin{equation}
\label{condition_2}
\|b\|^2=\sum\limits_{i=-\infty}^\infty|b_i|^2<\infty\,.
\end{equation}
Then the operator \(K\) with matrix \(K_{ij}=b_{i-j}V_{ij}\) is compact in \(H\).
\end{theorem}

\begin{proof}
Let us show first that the operator $K$ is bounded. It suffices to prove for that the estimates~\cite{9}
\begin{equation}
\label{bounded}
\forall i\in\mathbb{Z}_+\:\:\sum_{j=0}^\infty|K_{i,j}|<L\,,\quad
\forall j\in\mathbb{Z}_+\:\:\sum_{i=0}^\infty|K_{i,j}|<L\,,
\end{equation}
where $L$ is a constant independent of $i$ and $j$ ($\|K\|\le L$). Using Cauchy's inequality and ~(\ref{condition_2}), we have
$$
\sum_{j=0}^\infty|K_{i,j}|=\sum_{j=0}^\infty|b_{i-j}V_{i,j}|\le
\left(\sum_{j=0}^\infty |b_{i-j}|^2\right)^{1/2}
\left(\sum_{j=0}^\infty |V_{i,j}|^2\right)^{1/2}\le
$$
$$
\le\|b\|\sqrt{(VV^*)_{i,i}}\le \|b\|\cdot\|V\|\,.
$$
So the first estimate in~(\ref{bounded}) is fulfilled. By the same way the validity of second estimate in~(\ref{bounded}) is established. Thus the operator $K$ is bounded. We shall prove now its compactness.

Let's define the cut-off function $b^{(n)}=\{b^{(n)}_i\}_{i=-\infty}^\infty$ of the sequence $\{b_i\}$
$$
b_i^{(n)}=\left\{
                  \begin{array}{lc}
                  0\,,&|i|>n\\
                  b_i\,,&|i|\le n\\
                  \end{array}
          \right.
$$
and the sequence of operators $K^{(n)}$ by the formula $K^{(n)}_{i,j}=b^{(n)}_{i-j}V_{i,j}$. The matrix of the operator $K^{(n)}$ is band. It follows from Theorem~(\ref{compactness_for_band_matrix})  and condition~(\ref{condition_1}) that $K^{(n)}$ is a compact operator. We proceed further as in the proof of~(\ref{bounded}):
$$
\|K-K^{(n)}\|\le \|b-b^{(n)}\|\cdot\|V\|\,.
$$
Terefore
$$
\|K-K^{(n)}\|\to 0\,,\quad n\to\infty
$$
and $K$ is compact as a limit by norm of compact operators~\cite{9,10}.
\end{proof}

\section{Eigenvalues asymptotics of unbounded operators at the perturbation}
\label{main_theorems_about_asymptotics}

We start with the following simple lemma~\cite{12}.
\begin{lemma}
\label{eigenvalues_perturbation}
Let \(D\) be a normal operator in Hilbert space \(H\) with normalized eigenvectors \(e_n\:(\|e_n\|=1)\,,\,n\in\mathbb{N}\) and corresponding eigenvalues \(\mu_n\): \(D\,e_n=\mu_ne_n\). Let \(T=D+R\), where \(R\) is an arbitrary operator with dense definition domain \(\mathscr{D}(D)\) in \(H\) \((\,\mathscr{D}(D)\subseteq\mathscr{D}(R)\,)\) and $R^*e_n\in H$. Suppose the operator \(T\) has the normalized eigenvectors \(g_n\:(\|g_n\|=1)\,,\,n\in\mathbb{N}\) and corresponding eigenvalues \(\lambda_n\): \(Tg_n=\lambda_ng_n\). If $P_n$ and $Q_n$ are projectors on the eigenvectors $g_n$ and $e_n$ respectively and
\begin{equation}
\label{projectors_condition}
\|P_n-Q_n\|<1\,,
\end{equation}
then 
\begin{equation}
\label{formula_for_perturb_eigenvalues}
\lambda_n=\mu_n+\frac{(Re_n,e_n)}{\sqrt{1-\|P_n-Q_n\|^2}}+\frac{((g_n-e_n),R^*e_n)}
{\sqrt{1-\|P_n-Q_n\|^2}}\,.
\end{equation}
\end{lemma}

\begin{proof}
By definition $Tg_n=(D+R)g_n=\lambda_n g_n$. It follows that
$$
(Dg_n,e_n)+(Rg_n,e_n)=\lambda_n(g_n,e_n)
$$
or
$$
\mu_n(g_n,e_n)+(Re_n,e_n)+(R(g_n-e_n),e_n)=\lambda_n(g_n,e_n)\,.
$$
For the projectors $P_n=(\,\cdot\,,g_n)\,g_n$ and $Q_n=(\,\cdot\,,e_n)\,e_n$ one has the formula~\cite{6}
$$
\|P_n-Q_n\|^2=1-|(g_n,e_n)|^2\,.
$$
Hence, by virtue of (\ref{projectors_condition}), \((g_n,e_n)\ne 0\) and we can write
\begin{equation}
\label{lemma_formula1}
\lambda_n=\mu_n+\frac{(Re_n,e_n)}{(g_n,e_n)}+\frac{(R(g_n-e_n),e_n)}
{(g_n,e_n)}\,.
\end{equation}
We can always choose eigenvectors \(g_n\) and \(e_n\) such that $(g_n,e_n)\in\mathbb{R}$ and 
$(g_n,e_n)>0$. Then \((g_n,e_n)=\sqrt{1-\|P_n-Q_n\|^2}\) and the formula~(\ref{lemma_formula1}) yields
$$
\lambda_n=\mu_n+\frac{(Re_n,e_n)}{\sqrt{1-\|P_n-Q_n\|^2}}+\frac{((g_n-e_n),R^*e_n)}
{\sqrt{1-\|P_n-Q_n\|^2}}\,.
$$
The operator \(R^*\) exists as the domain of \(R\) is dense in \(H\).
\end{proof}

Let us note that the conditions of this lemma do not require neither discreteness of operators spectra no the boundedness of perturbation \(R\). It is important only that $Re_n,R^*e_n\in H$. Besides the corresponding eigenvectors and eigenvalues have an arbitrary numeration.

\begin{corollary}
If, in addition to the lemma's conditions, \(R^*\) is relatively compact with respect to \(D\) and \(\|P_n-Q_n\|\le q<1\), \(n\in\mathbb{N}\), then
$$
\lambda_n=\mu_n\left(1+O\left(\|R^*D^{-1}e_n\|\right)\right)\,,\quad n\to \infty\,.
$$
\end{corollary}
\begin{proof}
Since the operator \(R^*\) is relatively compact with respect to \(D\), there exist the bounded operator \(D^{-1}\) and the operator \(R^*D^{-1}\) is compact such that \(\mathscr{D}(D)\subseteq \mathscr{D}(R^*)\). From~(\ref{formula_for_perturb_eigenvalues}) it follows that
$$
\lambda_n=\mu_n+\frac{(e_n,R^*D^{-1}De_n)}{\sqrt{1-\|P_n-Q_n\|^2}}+\frac{((g_n-e_n),R^*D^{-1}De_n)}
{\sqrt{1-\|P_n-Q_n\|^2}}=
$$
$$
=\mu_n+\bar \mu_n\frac{(e_n,R^*D^{-1}e_n)}{\sqrt{1-\|P_n-Q_n\|^2}}+\bar \mu_n\frac{((g_n-e_n),R^*D^{-1}e_n)}{\sqrt{1-\|P_n-Q_n\|^2}}=\mu_n(1+c_n)\,,
$$
$$
|c_n|\le \frac{\|R^*D^{-1}e_n\|}{\sqrt{1-q^2}}+\frac{\|g_n-e_n\|\cdot\|R^*D^{-1}e_n\|}{\sqrt{1-q^2}}\,.
$$
As \((g_n,e_n)>0\), we have \(\|g_n-e_n\|\le\sqrt{2}\|P_n-Q_n\|\le q\) and
$$
c_n=O\left(\|R^*D^{-1}e_n\|\right)\,.
$$
By Lemma~(\ref{lemma_compact}) \(\|R^*D^{-1}e_n\|\to 0\) as the eigenvectors of normal operator form orthonormal system.
\end{proof}

From this lemma it follows that the asymptotics of perturbed eigenvalues at \(n\to\infty\) essentially depends on the behaviour of \(\|P_n-Q_n\|\). In the work~\cite{5} it was shown that if the perturbation \(R\) is compact and \(|\mu_i-\mu_k|\ge\epsilon_0>0\,,\:\forall i\ne k\), then \(\|P_n-Q_n\|<1\) for all sufficiently large \(n\). We shall show that in this case in fact \(\|P_n-Q_n\|\to 0\) as \(n\to\infty\). Besides, we shall consider arbitrary perturbations and unperturbed eigenvalues without non accumulative restriction. To avoid technical complications and for the next applications we consider the case when the operators \(T\) and \(D\) are selfadjoint.

\begin{theorem}
\label{eigenvalues_at_relatively_compact_R}
Let \(D\) be a selfadjoint operator in Hilbert space \(H\) with orthonormal eigenvectors \(e_n\,,\,n\in\mathbb{N}\) formed basis in \(H\) and corresponding eigenvalues \(\mu_n\) \((\,\mu_1<\mu_2<\ldots<\mu_n\ldots\,,\,\lim \mu_n=+\infty\,)\). Let \(R\) be either symmetric operator relatively compact with respect to \(D\) or arbitrary compact operator. Denote
$$
\Delta_n^2\equiv\sum\limits_{m=1}^{\infty}\frac{\|Re_m\|^2}{\|(D-\mu_nE-r_nE_n)e_m\|^2}\,,
$$
where  \(r_n=\dfrac{1}{2}\,\mbox{min}\{\mu_n-\mu_{n-1}\,;\,\mu_{n+1}-\mu_n\}\) and the operator \(E_n\) is defined by
$$
E_ne_m=\left\{\begin{array}{r}
e_m\,,\:m>n\\
-e_m\,,\:m\le n
\end{array}\right.\,.
$$
Suppose \(\Delta_n\) exist and \(\Delta_n\to 0\). Then for the eigenvalues \(\lambda_n\) of the operator \(T=D+R\) (having discrete spectrum too) we have the following asymptotic formula
$$
\lambda_n=\mu_n+(Re_n,e_n)+O\left(\Delta_n\|R^*e_n\|\right)\,,\quad n\to\infty\,.
$$
\end{theorem}

\begin{proof}
First note that the operator \(T\) has a discrete spectrum as the operator \(D\)~\cite{10,13}. If \(R\) is symmetric and relatively compact with respect to \(D\) then \(T\) is selfadjoint. We start with a known resolvent identity
$$
(T-\lambda E)^{-1}=(D-\lambda E)^{-1}\left(E+R\,(D-\lambda E)^{-1}\right)^{-1}\,.
$$
If \(\lambda\notin\sigma(D)\) and \(\|R\,(D-\lambda E)^{-1}\|<1\), this formula defines a resolvent \((T-\lambda E)^{-1}\) as a convergent by norm operator series
\begin{equation}
\label{resolvent_series}
(T-\lambda E)^{-1}=(D-\lambda E)^{-1}\left(E+\sum\limits_{k=1}^\infty(-1)^k\left[R\,(D-\lambda E)^{-1}\right]^k\right)\,.
\end{equation}
Let \(\lambda\in C_n\), where \(C_n\) is a circle: \(C_n=\{\lambda\in\mathbb{C}\,:\,|\lambda-\mu_n|=r_n\}\) with some radius \(r_n<\mbox{dist}(\lambda_n\,,\,\sigma(D)\setminus\lambda_n))\).
For any \(f\in H\) one has
$$
(D-\lambda E)^{-1}f=\sum\limits_{m=1}^\infty\frac{(f,e_m)\,e_m}{\mu_m-\lambda}\quad (\lambda\in C_n)\,.
$$
and
$$
R\,(D-\lambda E)^{-1}f=\sum\limits_{m,k=1}^\infty\frac{(f,e_m)\,(Re_m,e_k)\,e_k}{\mu_m-\lambda}\,,
$$
so that
$$
\|R\,(D-\lambda E)^{-1}f\|^2=\sum\limits_{k=1}^\infty\left|\sum\limits_{m=1}^\infty\frac{(f,e_m)\,(Re_m,e_k)}{\mu_m-\lambda}\right|^2\le\sum\limits_{k=1}^\infty\sum\limits_{m=1}^\infty\frac{|(Re_m,e_k)|^2}{|\mu_m-\lambda|^2}\,\|f\|^2\,.
$$
Therefore
$$
\|R\,(D-\lambda E)^{-1}\|^2\le\sum\limits_{k=1}^\infty\sum\limits_{m=1}^\infty\frac{|(Re_m,e_k)|^2}{|\mu_m-\lambda|^2}=\sum\limits_{m=1}^\infty\frac{\|Re_m\|^2}{|\mu_m-\lambda|^2}\quad (\lambda\in C_n)\,.
$$
Consider the sum in the right hand side of this inequality. We have for \(\lambda\in C_n\) 
$$
\sum\limits_{m=1}^\infty\frac{\|Re_m\|^2}{|\mu_m-\lambda|^2}=\sum\limits_{m=1}^{n-1}\frac{\|Re_m\|^2}{|\mu_m-\lambda|^2}+\frac{\|Re_n\|^2}{r_n^2}+\sum\limits_{m=n+1}^\infty\frac{\|Re_m\|^2}{|\mu_m-\lambda|^2}\le 
$$
$$
\le\sum\limits_{m=1}^{n-1}\frac{\|Re_m\|^2}{(\mu_m-\mu_n+r_n)^2}+\frac{\|Re_n\|^2}{r_n^2}+\sum\limits_{m=n+1}^\infty\frac{\|Re_m\|^2}{(\mu_m-\mu_n-r_n)^2}=
$$
$$
=\sum\limits_{m=1}^{\infty}\frac{\|Re_m\|^2}{\|(D-\mu_nE-r_nE_n)e_m\|^2}=\Delta_n^2\,.
$$
Putting \(r_n=\dfrac{1}{2}\,\mbox{min}\{\mu_n-\mu_{n-1}\,;\,\mu_{n+1}-\mu_n\}\), we obtain for \(\lambda\in C_n\)
$$
\|R\,(D-\lambda E)^{-1}\|\le\Delta_n
$$
and \(\Delta_n\to 0\) because of assumption of the theorem. Therefore for all sufficiently large \(n\) the series~(\ref{resolvent_series}) converges and defines the resolvent \((T-\lambda E)^{-1}\). Hence one can define the Riesz projector
$$
P_n=-\frac{1}{2\pi i}\oint\limits_{C_n}(T-\lambda E)^{-1}\,d\lambda=Q_n-\frac{1}{2\pi i}\sum\limits_{k=1}^\infty(-1)^k\oint\limits_{C_n}(D-\lambda E)^{-1}\left[R\,(D-\lambda E)^{-1}\right]^k\,d\lambda\,,
$$
where
$$
Q_n=-\frac{1}{2\pi i}\oint\limits_{C_n}(D-\lambda E)^{-1}\,d\lambda=(\,\cdot\,,e_n)\,e_n\,.
$$
It follows that
$$
\|P_n-Q_n\|\le\frac{1}{2\pi}\sum\limits_{k=1}^\infty\oint\limits_{C_n}\|(D-\lambda E)^{-1}\|\cdot\left\|R\,(D-\lambda E)^{-1}\right\|^k\,d\lambda\le\frac{1}{2\pi}\sum\limits_{k=1}^\infty
\frac{\Delta_n^k}{r_n}\oint\limits_{C_n}\,d\lambda=\frac{\Delta_n}{1-\Delta_n}
$$
whence
$$
\|P_n-Q_n\|=O\left(\Delta_n\right)\,,\quad n\to\infty\,.
$$
Hence  \(\|P_n-Q_n\|<1\) for all sufficiently large \(n\) and \(\mbox{rank}\,P_n=
\mbox{rank}\,Q_n=1\). It follows that for all \(n\) large enough inside the circle \(C_n\) there is exactly one eigenvalue \(\lambda_n\) of the operator \(T\). From lemma~(\ref{eigenvalues_perturbation}) we obtain as \(n\to\infty\)
$$
\lambda_n=\mu_n+(Re_n,e_n)+(Re_n,e_n)\cdot O\left(\Delta_n^2\right)+O\left(\Delta_n\|R^*e_n\|\right)=\mu_n+(Re_n,e_n)+O\left(\Delta_n\|R^*e_n\|\right)\,.
$$
\end{proof}

\begin{remark}
\label{R_by_R*}
When \(R\) is arbitrary compact, in the definition of \(\Delta_n\) one can replace \(R\) by \(R^*\). 
\end{remark}
\begin{proof}
If we start from the identity
$$
(T-\lambda E)^{-1}=\left(E+(D-\lambda E)^{-1}R\right)^{-1}(D-\lambda E)^{-1}\,,
$$
we come to the series
$$
(T-\lambda E)^{-1}=\left(E+\sum\limits_{k=1}^\infty(-1)^k\left[(D-\lambda E)^{-1}R\right]^k\right)(D-\lambda E)^{-1}
$$
instead of (\ref{resolvent_series}). Since \(\|(D-\lambda E)^{-1}R\|=\|R^*(D-\bar{\lambda} E)^{-1}\|\) according to the property~(\ref{Symmetric_property}), by the same way we will obtain
$$
\|(D-\lambda E)^{-1}R\|\le\Delta_n\,,
$$
where in the definition of \(\Delta_n\) we will have \(R^*\) instead of \(R\).
\end{proof}

From the expression for \(\Delta_n\)
$$
\Delta_n^2=\sum\limits_{m=1}^{n-1}\frac{\|Re_m\|^2}{(\mu_m-\mu_n+r_n)^2}+\frac{\|Re_n\|^2}{r_n^2}+\sum\limits_{m=n+1}^\infty\frac{\|Re_m\|^2}{(\mu_m-\mu_n-r_n)^2}
$$
it follows that \(\Delta_n\to 0\) implies \(\dfrac{\|Re_n\|}{r_n}\to 0\) or \(\dfrac{\|Re_n\|}{\mu_{n+1}-\mu_n}\to 0\). So in order that \(\Delta_n\to 0\) for any compact operator \(R\) it is necessary that \(\forall n\in\mathbb{N}\:\:\mu_{n+1}-\mu_n\ge\epsilon_0>0\). Let us show that this condition is also sufficient.

Suppose that the operator \(R\) is compact and \(\forall n\in\mathbb{N}\:\:\mu_{n+1}-\mu_n\ge\epsilon_0>0\). We can put \(r_n=\epsilon_0/2\). One has for \(i>j\)
$$
\mu_i-\mu_j\ge(i-j)\,\epsilon_0
$$
and
$$
\Delta_n^2\le\sum\limits_{m=1}^{n-1}\frac{\|Re_m\|^2}{((n-m)\,\epsilon_0-\epsilon_0/2)^2}+\frac{4\,\|Re_n\|^2}{\epsilon_0^2}+\sum\limits_{m=n+1}^\infty\frac{\|Re_m\|^2}{((m-n)\,\epsilon_0-\epsilon_0/2)^2}
$$
or
$$
\epsilon_0^2\,\Delta_n^2\le\sum\limits_{m=1}^{n-1}\frac{\|Re_m\|^2}{(n-m-1/2)^2}+4\,\|Re_n\|^2+\sum\limits_{m=n+1}^\infty\frac{\|Re_m\|^2}{(m-n-1/2)^2}\,.
$$
Both sums tends to zero at \(n\to\infty\). First sum tends to zero by the theorem on regular transformations~\cite{8}. Actually we have
$$
t_n=\sum\limits_{m=1}^{n-1}\frac{\|Re_m\|^2}{(n-m-1/2)^2}=\sum\limits_{m=1}^\infty C_{n,m}\,s_m\,,
$$
where
$$
s_m=\|Re_m\|^2\,,\quad C_{n,m}=\left\{\begin{array}{r}
(n-m-1/2)^{-2}\,,\,m<n\\
0\quad\qquad\,,\,m\ge n
\end{array}\right.\,.
$$
Since
$$
\sum\limits_{m=1}^\infty |C_{n,m}|=\sum\limits_{k=1}^{n-1}\frac{1}{(k-1/2)^2}=4\sum\limits_{k=1}^{n-1}\frac{1}{(2k-1)^2}<4\sum\limits_{k=1}^{\infty}\frac{1}{(2k-1)^2}=\frac{\pi^2}{2}\,,
$$
$$
\lim\limits_{n\to\infty}C_{n,m}=0\,,\quad m\in\mathbb{N}\,,
$$
the transformation \(C_{n,m}\) is regular and from \(s_m\to 0\) it follows that \(t_n\to 0\).

For a second sum one has
$$
q_n=\sum\limits_{m=n+1}^\infty\frac{\|Re_m\|^2}{(m-n-1/2)^2}=\sum\limits_{k=1}^\infty\frac{s_{k+n}}{(k-1/2)^2}\,.
$$
For any \(\epsilon>0\) there exists \(N\in\mathbb{N}\) such that \(s_n<2\,\epsilon/\pi^2\) for all \(n>N\).
For such values of \(n\) we have
$$
q_n<\frac{2\,\epsilon}{\pi^2}\,\sum\limits_{k=1}^\infty\frac{1}{(k-1/2)^2}=\epsilon\,.
$$
Therefore \(q_n\to 0\) and \(\Delta_n\to 0\). Thus we obtain

\begin{theorem}
\label{delta_to_zero}
In order that \(\Delta_n\to 0\) for any compact operator \(R\) it is necessary and sufficient that \(\forall n\in\mathbb{N}\:\:\mu_{n+1}-\mu_n\ge\epsilon_0>0\).
\end{theorem}

However there exist operators \(R\) (compact or not) such that \(\Delta_n\to 0\) at \((\mu_{n+1}-\mu_n)\to 0\).\\
Consider the case \(\mu_n=n^2\) (see also \cite{12}).  In this case \(\mu_{n+1}-\mu_n=2n+1\) and \(r_n=n-1/2\). One has
$$
\Delta_n^2=\sum\limits_{m=1}^{n-1}\frac{\|Re_m\|^2}{(n^2-m^2-n+1/2)^2}+\frac{\|Re_n\|^2}{(n-1/2)^2}+\sum\limits_{m=n+1}^\infty\frac{\|Re_m\|^2}{(m^2-n^2-n+1/2)^2}=
$$
$$
=\sum\limits_{m=1}^{n-1}\frac{\|Re_m\|^2}{((n-1/2)^2-m^2+1/4)^2}+\frac{\|Re_n\|^2}{(n-1/2)^2}+\sum\limits_{m=n+1}^\infty\frac{\|Re_m\|^2}{(m^2-(n+1/2)^2+3/4)^2}<
$$
$$
<\sum\limits_{m=1}^{n-1}\frac{\|Re_m\|^2}{((n-1/2)^2-m^2)^2}+\frac{\|Re_n\|^2}{(n-1/2)^2}+\sum\limits_{m=n+1}^\infty\frac{\|Re_m\|^2}{(m^2-(n+1/2)^2)^2}=
$$
$$
=\sum\limits_{m=1}^{n-1}\frac{\|Re_m\|^2}{(n-1/2-m)^2\,(n-1/2+m)^2}+\frac{\|Re_n\|^2}{(n-1/2)^2}+\sum\limits_{m=n+1}^\infty\frac{\|Re_m\|^2}{(m-n-1/2)^2\,(m+n+1/2)^2}<
$$
$$
<\frac{1}{(n-1/2)^2}\left(\sum\limits_{m=1}^{n-1}\frac{\|Re_m\|^2}{(n-1/2-m)^2}+\|Re_n\|^2+\sum\limits_{m=n+1}^\infty\frac{\|Re_m\|^2}{(m-n-1/2)^2}\right)=
$$
$$
=\frac{1}{(n-1/2)^2}\left(\sum\limits_{k=1}^{n-1}\frac{\|Re_{n-k}\|^2}{(k-1/2)^2}+\|Re_n\|^2+\sum\limits_{k=1}^\infty\frac{\|Re_{k+n}\|^2}{(k-1/2)^2}\right)\,.
$$
Suppose \(\|Re_n\|=O\left(n^\alpha\right)\), \(\alpha<1/2\). Then
$$
\Delta_n^2<\frac{C}{(n-1/2)^2}\left(\sum\limits_{k=1}^{n-1}\frac{(n-k)^{2\alpha}}{(k-1/2)^2}+n^{2\alpha}+\sum\limits_{k=1}^\infty\frac{(k+n)^{2\alpha}}{(k-1/2)^2}\right)<
$$
$$
<\frac{C\,n^{2\alpha}}{(n-1/2)^2}\left(\sum\limits_{k=1}^{n-1}\frac{1}{(k-1/2)^2}+1+\sum\limits_{k=1}^\infty\frac{(k+1)^{2\alpha}}{(k-1/2)^2}\right)\,.
$$
Therefore \(\Delta_n=O\left(n^{\alpha-1}\right)\). If also \(\|R^*e_n\|=O\left(n^\alpha\right)\), then the theorem~(\ref{eigenvalues_at_relatively_compact_R}) gives
$$
\lambda_n=n^2+(Re_n,e_n)+O\left(\frac{1}{n^{1-2\alpha}}\right)\,,\:n\to\infty\,.
$$
In~\cite{12} this formula was obtained at \(\alpha=0\).

Let's obtain another result which we shall use in the next section. The following theorem is a generalization of the corresponding result from~\cite{11}.

\begin{theorem}
\label{noncompact_perturbation_oscillator}
Let \(D\) be a selfadjoint operator in Hilbert space \(H\) with orthonormal eigenvectors \(e_n\,,\,n\in\mathbb{N}\) formed basis in \(H\) and corresponding eigenvalues \(\mu_n=n\). Let \(R\) be a bounded selfadjoint non-compact operator with matrix \(R_{ij}=(Re_j,e_i)\) such that
$$
\lim\limits_{n\to\infty}R_{n,n+p}=0\,,\quad p\in\mathbb{Z}\,.
$$
Then the eigenvalues \(\lambda_n\) of selfadjoint operator \(T=D+R\) (having discrete spectrum too) satisfy the following asymptotic formula
$$
\lambda_n=n+R_{nn}+\sum\limits_{k\ne n}\frac{|R_{nk}|^2}{n-k}+O\left((\Delta_n+\|Ke_n\|)\,(|R_{nn}|+\|Ke_n\|)\right)\,,
$$
where
$$
\|Ke_n\|^2=\sum_{k\ne n}\frac{|R_{kn}|^2}{(k-n)^2}\,,\quad\Delta_n^2=O\left(\sum\limits_{m=1}^{\infty}\frac{(|R_{mm}|+\|Ke_m\|)^2}{\|(D-nE-E_n/2)e_m\|^2}\right)\,.
$$
\end{theorem}
\begin{proof}
Let us show that there exists an anti-hermitian operator $K$ ($K^*=-K$) such that
\begin{equation}
\label{equivalent}
(E+K)T-D(E+K)=B\,,
\end{equation}
where $B$ is compact operator. This condition means that
$$
T=(E+K)^{-1}(D+B(E+K)^{-1})(E+K)\,.
$$
The inverse operator $(E+K)^{-1}$ exists because $K$ is anti-hermitian (spectrum of \(K\) lies on imaginary axis). Thus the operators $T$ and $D+B(E+K)^{-1}$ are similar and have the same spectrum. If the operator $B$ is compact then, according to the first property~(\ref{Ideal_property}), the operator \(B(E+K)^{-1}\) is also compact. Then from theorems~(\ref{eigenvalues_at_relatively_compact_R}) and~(\ref{delta_to_zero}) we obtain
$$
\lambda_n=n+(B(E+K)^{-1}e_n,e_n)+O\left(\Delta_n\|(E+K^*)^{-1}B^*e_n\|\right)
$$
or
\begin{equation}
\label{t2_3}
\lambda_n=n+(B(E+K)^{-1}e_n,e_n)+O\left(\Delta_n\|B^*e_n\|\right)\,,
\end{equation}
From~(\ref{equivalent}) it follows that
\begin{equation}
\label{t2_4}
B=(E+K)(D+R)-D(I+K)=R+KD-DK+KR=R-[D,K]+KR\,.
\end{equation}
Denote \(R'=R-\mbox{diag}\{R_{nn}\}\) so that \(R'_{nn}=0\) and \(R'_{nm}=R_{nm}\) for \(n\ne m\). Suppose that $[D,K]=R'$ or in matrix form $K_{ij}\,(i-j)=R'_{ij}$. It follows that
\begin{equation}
\label{t2_5}
K_{ij}=\frac{R_{ij}}{i-j}\,,\quad i\ne j\,.
\end{equation}
Let us put \(K_{ii}=0\), \(i\in\mathbb{N}\). As the operator $R$ is selfadjoint the operator $K$ is anti-hermitian. Besides the operator \(K\) is compact. It follows from theorem~(\ref{compactness_for_nonband_matrix}) if we put \(b_i=1/i\) ($i\ne0$).

Now from~(\ref{t2_4}) we find
\begin{equation}
\label{formula_for_compact_B}
B=\mbox{diag}\{R_{nn}\}+KR\,,
\end{equation}
and \(B\) is compact.

Let us look for the operator \((E+K)^{-1}\) in the form \((E+K)^{-1}=E-K+X\). From the identity
$$
E=(E+K)\,(E+K)^{-1}=(E+K)\,(E-K+X)=E-K^2+(E+K)X
$$
we find \(X=(E+K)^{-1}K^2\) and \((E+K)^{-1}=E-K+(E+K)^{-1}K^2\). Using~(\ref{formula_for_compact_B}) we obtain
$$
B(E+K)^{-1}=\mbox{diag}\{R_{nn}\}+KR-\mbox{diag}\{R_{nn}\}\,K-KRK+B(E+K)^{-1}K^2\,.
$$
From this formula it follows that
$$
(B(E+K)^{-1}e_n,e_n)=R_{nn}+(KR\,e_n,e_n)-(KRKe_n,e_n)+(B(E+K)^{-1}K^2e_n,e_n)\,,
$$
as \((\mbox{diag}\{R_{nn}\}\,Ke_n,e_n)=0\). Using~(\ref{t2_5}) we obtain
$$
(B(E+K)^{-1}e_n,e_n)=R_{nn}+\sum\limits_{k\ne n}\frac{|R_{nk}|^2}{n-k}+(RKe_n,Ke_n)+((E+K)^{-1}K^2e_n,B^*e_n)=
$$
$$
=R_{nn}+\sum\limits_{k\ne n}\frac{|R_{nk}|^2}{n-k}+O\left(\|Ke_n\|^2\right)+O\left(\|Ke_n\|\,(|R_{nn}|+\|Ke_n\|)\right)=
$$
$$
=R_{nn}+\sum\limits_{k\ne n}\frac{|R_{nk}|^2}{n-k}+O\left(\|Ke_n\|\,(|R_{nn}|+\|Ke_n\|)\right)\,.
$$
Here we used that \(B^*=\mbox{diag}\{R_{nn}\}-RK\). Substituting this expression in the formula~(\ref{t2_3}) we obtain
$$
\lambda_n=n+R_{nn}+\sum\limits_{k\ne n}\frac{|R_{nk}|^2}{n-k}+O\left(\|Ke_n\|\,(|R_{nn}|+\|Ke_n\|)\right)+O\left(\Delta_n\|B^*e_n\|\right)=
$$
$$
=n+R_{nn}+\sum\limits_{k\ne n}\frac{|R_{nk}|^2}{n-k}+O\left((\Delta_n+\|Ke_n\|)\,(|R_{nn}|+\|Ke_n\|)\right)\,.
$$
From~(\ref{t2_5}) we find
$$
\|Ke_n\|^2=\sum_{k\ne n}\frac{|R_{kn}|^2}{(k-n)^2}\,.
$$
Taking into account the remark (\ref{R_by_R*}), we obtain
$$
\Delta_n^2=\sum\limits_{m=1}^{\infty}\frac{\|(E-K)^{-1}B^*e_m\|^2}{\|(D-nE-E_n/2)e_m\|^2}=
O\left(\sum\limits_{m=1}^{\infty}\frac{(|R_{mm}|+\|Ke_m\|)^2}{\|(D-nE-E_n/2)e_m\|^2}\right)\,.
$$
\end{proof}

\begin{remark}
Note that \(R_{nn}\) and \(\sum\limits_{k\ne n}\dfrac{|R_{nk}|^2}{n-k}\) are the first-order and second-order terms of perturbation theory series for \(\lambda_n\) respectively~{\rm \cite{18,19}}.
\end{remark}

\section{Two-photon quantum Rabi model}
\label{Rabi}

The energy operator (Hamiltonian) of the two-level quantum Rabi model has the form~\cite{1,2,3,4}
\begin{equation}
\label{full_hamiltonian_Rabi}
  \hat {\bf H}=\frac{\Delta}{2}\,\hat{\bf\sigma}_z+
  \hat {\bf a}^+\,\hat {\bf a}+g\,(\,\hat{\bf a}^2+\hat {\bf a}^{+2}\,)\,\hat{\bf\sigma}_x\,,
\end{equation}
where $\hat{\bf\sigma}_z,\hat{\bf\sigma}_x$ are the $2\times2$
matrices of form
$$
\hat{\bf\sigma}_z=\left(\begin{array}{cc} 1 & 0\\0 & -1\end{array}\right)
\,,\quad
\hat{\bf\sigma}_x=\left(\begin{array}{cc} 0 & 1\\1 & 0\end{array}\right)
\,,
$$
$\hat{\bf a}$ and $\hat{\bf a}^+$ are the creation and annihilation operators of the harmonic oscillator satisfying the commutative relation \([\hat{\bf a},\hat{\bf a}^+]=1\) , $g\in(0,1/2)$ is the interaction constant, \(\Delta\in\mathbb{R}\) is the transition frequency in the two-level system. The condition $g\in(0,1/2)$ insures a discrete spectrum of the operator \(\hat {\bf H}\). At the beginning of this section we will use Dirac's notations.

The operator \(\hat {\bf H}\) acts in the Hilbert space \(H=H_1\otimes H_2\), where \(H_1=\{\,|+\rangle,|-\rangle\,\}\) is a Hilbert space of the two-level system ( \(\hat{\bf\sigma}_z\,|\pm\rangle=\pm\,|\pm\rangle\) ) and \(H_2=\{\,|n\rangle\,\}_{n=0}^\infty\) is a Hilbert space of oscillator: \( \hat {\bf a}^+\,\hat {\bf a}\:|n\rangle=n\,|n\rangle\).
So the Hilbert space \(H\) consists of vectors \(|n,\pm\rangle=|n\rangle|\pm\rangle\), \(n\in\mathbb{Z}_+\). One can partition \(H\) into the sum of two subspaces \(H=H'_1\oplus H'_2\), where
$$
H'_1=\{\,|0,+\rangle, |1,+\rangle, |2,-\rangle, |3,-\rangle, |4,+\rangle,|5,+\rangle, |6,-\rangle, |7,-\rangle, \ldots\}\,,
$$
$$
H'_2=\{\,|0,-\rangle, |1,-\rangle, |2,+\rangle, |3,+\rangle, |4,-\rangle,|5,-\rangle, |6,+\rangle, |7,+\rangle, \ldots\}\,.
$$
It is easy to check that \(H'_1\) and \(H'_2\) are invariant subspaces for the operator \(\hat{\bf H}\). In the bases of this subspaces vectors the operator \(\hat{\bf H}\) generates two operators of the form
$$
\hat{\bf H}_i=\hat {\bf a}_i^+\,\hat {\bf a}_i+g\,(\,\hat{\bf a}_i^2+\hat {\bf a}_i^{+2}\,)+\hat{\bf V}_i\,,\quad i=1,2\,,
$$
where $\hat{\bf a}_i$ and $\hat{\bf a}_i^+$ are new creation and annihilation operators and 
\(\hat{\bf V}_i\) are diagonal operators of the form
$$
\hat{\bf V}_1=\frac{\Delta}{2}\,\mbox{diag}\left\{\,1,\,1,-1,-1,\,1,\,1,\,\ldots\,\right\}\,,
$$
$$
\hat{\bf V}_2=\frac{\Delta}{2}\,\mbox{diag}\left\{-1,-1,\,1,\,1,-1,-1,\,\ldots\,\right\}=-\hat{\bf V}_1\,.
$$

Coming back to the mathematical notations we obtain for consideration the operator of the form
\begin{equation}
\label{main_operator_two_Rabi_maodel}
H=H_0+V\,,\quad H_0=a^+a+g\,(a^{2}+a^{+2})\,,
\end{equation}
where \(V\) is a diagonal periodic matrix of the form \(\hat{\bf V}_1\) or \(-\hat{\bf V}_1\). Eigenvalues problem for the operator \(H_0\) can be solved explicitly~\cite{2,3}. Let's consider for 
\(\lambda\in\mathbb{R}\) the orthogonal "contraction" operator \(U\)~\cite{4}
$$
U=e^{\textstyle\lambda\,(a^2-a^{+2})}\,,\quad U^+=U^T=e^{\textstyle-\lambda\,(a^2-a^{+2})}\,,\quad UU^+=U^+U=E\,.
$$
( For \(\lambda\in\mathbb{C}\) one can consider the unitary operator \(U=e^{\textstyle\lambda\,a^2-\bar{\lambda}\,a^{+2}}\). But here we won't need it. )

\begin{lemma}[Hadamard~\cite{4}]
\label{Hadamard_lemma}
For any operators \(A\) and \(B\) the formula
$$
e^A\,B\,e^{-A}=B+\frac{[A,B]}{1!}+\frac{[A,[A,B]]}{2!}+\frac{[A,[A,[A,B]]]}{3!}+\ldots
$$
is valid if the series converges in some operator sense. 
\end{lemma}

Let \(A=\lambda\,(a^{+2}-a^2)\). Then taking into account that \([a^{+2},a]=-2a^+\), \([a^2,a^+]=2a\), one has
$$
[A,a]=\lambda\,[a^{+2},a]=-2\lambda\,a^+\,,
$$
$$
[A,[A,a]]=2\lambda^2\,[a^2,a^+]=4\lambda^2\,a\,,
$$
$$
[A,[A,[A,a]]]=4\lambda^3\,[a^{+2},a]=-(2\lambda)^3a^+\,,\,\ldots \,\mbox{etc}.
$$
Using Hadamard's lemma~(\ref{Hadamard_lemma}) we obtain
\begin{equation}
\label{transformation_a_by_U}
U^+\,a\,U=a-\frac{2\lambda}{1!}\,a^++\frac{(2\lambda)^2}{2!}\,a-\frac{(2\lambda)^3}{3!}\,a^++\ldots=
\mbox{ch}(2\lambda)\,a-\mbox{sh}(2\lambda)\,a^+
\end{equation}
and
$$
U^+\,H_0\,U=\left(\mbox{ch}(2\lambda)\,a^+-\mbox{sh}(2\lambda)\,a\right)\left(\mbox{ch}(2\lambda)\,a-\mbox{sh}(2\lambda)\,a^+\right)+
$$
$$
+g\left(\left(\mbox{ch}(2\lambda)\,a-\mbox{sh}(2\lambda)\,a^+\right)^2+\left(\mbox{ch}(2\lambda)\,a^+-\mbox{sh}(2\lambda)\,a\right)^2\right)\,.
$$
After simple transformations one has
$$
U^+\,H_0\,U=\left(\mbox{ch}^2(2\lambda)-2g\,\mbox{sh}(2\lambda)\,\mbox{ch}(2\lambda)\right)a^+a+
\left(\mbox{sh}^2(2\lambda)-2g\,\mbox{sh}(2\lambda)\,\mbox{ch}(2\lambda)\right)a\,a^+ +
$$
$$
+\left(g\,(\mbox{ch}^2(2\lambda)+\mbox{sh}^2(2\lambda))-\mbox{sh}(2\lambda)\,\mbox{ch}(2\lambda)\right)
\left(a^2+a^{+2}\right)\,.
$$
If \(g\,(\mbox{ch}^2(2\lambda)+\mbox{sh}^2(2\lambda))=\mbox{sh}(2\lambda)\,\mbox{ch}(2\lambda)\) or
\begin{equation}
\label{equation_for_lambda}
\mbox{th}(4\lambda)=2g
\end{equation}
then the transformed operator \(U^+\,H_0\,U\) becomes the oscillator operator. Since \(0<g<1/2\) the equation~(\ref{equation_for_lambda}) has always a solution for \(\lambda\). One has next
$$
U^+\,H_0\,U=\left(\mbox{ch}^2(2\lambda)-4g\,\mbox{sh}(2\lambda)\,\mbox{ch}(2\lambda)+\mbox{sh}^2(2\lambda)\right)a^+a+\mbox{sh}^2(2\lambda)-2g\,\mbox{sh}(2\lambda)\,\mbox{ch}(2\lambda)=
$$
$$
=\left(\mbox{ch}(4\lambda)-2g\,\mbox{sh}(4\lambda)\right)a^+a+\frac{\mbox{ch}(4\lambda)-1}{2}-g\,\mbox{sh}(4\lambda)\,.
$$
Taking into account~(\ref{equation_for_lambda}) we obtain finally
$$
\tilde{H}_0=U^+\,H_0\,U=\omega\left(a^+a+\frac{1}{2}\right)-\frac{1}{2}\,,
$$
where \(\omega=\sqrt{1-4g^2}\). If \(\{e_n\}_{n=0}^\infty\) is a orthonormal basis of matrix representation of~(\ref{main_operator_two_Rabi_maodel}), then
\begin{equation}
\label{diagonal_H_0}
\tilde{H}_0\,e_n=E^{(0)}_n\,e_n\,,\quad E^{(0)}_n=\omega\left(n+\frac{1}{2}\right)-\frac{1}{2}\,,\quad n\in\mathbb{Z}_+\,,\quad \omega=\sqrt{1-4g^2}
\end{equation}
or
$$
H_0\,a_n=E^{(0)}_n\,a_n\,,\quad a_n=U\,e_n\,.
$$
If we put in lemma~(\ref{Hadamard_lemma}) \(A=x\,a^+a\) ( \(x\in\mathbb{R}\) ), \(B=a^+\), then one has 
$$
[A,B]=x\,[a^+a,a^+]=x\,a^+\,,\quad [A,[A,B]]=x^2\,a^+\,,\quad [A,[A,[A,B]]]=x^3\,a^+\,,\:\ldots\mbox{etc}\,.
$$
Hence
\begin{equation}
\label{formula1}
e^{\textstyle x\,a^+a}\,a^+\,e^{\textstyle -x\,a^+a}=a^+\,e^{x}\,.
\end{equation}
By the same way we obtain
\begin{equation}
\label{formula2}
e^{\textstyle -x\,a^+a}\,a\,e^{\textstyle x\,a^+a}=a\,e^{x}\,.
\end{equation}

We will need further the following lemma~\cite{4}
\begin{lemma}
\label{differentiation_function_of_a_and_a+}
Let \(f(x,y)\) be an arbitrary function expanded in a power series and
$$
\frac{\partial f(a,a^+)}{\partial a}\equiv\lim\limits_{\Delta v\to0}\frac{f(a+\Delta v,a^+)-f(a,a^+)}{\Delta v}\,,\quad \frac{\partial f(a,a^+)}{\partial a^+}\equiv\lim\limits_{\Delta v\to0}\frac{f(a,a^++\Delta v)-f(a,a^+)}{\Delta v}\,.
$$
Then
$$
\frac{\partial f(a,a^+)}{\partial\,a}=-[a^+,f(a,a^+)]\,,\quad 
\frac{\partial f(a,a^+)}{\partial\,a^+}=[a,f(a,a^+)]\,.
$$
\end{lemma}

For example, putting \(f(a,a^+)=e^{\textstyle\beta\,a^+a}\) and taking into account~(\ref{formula1}) and~(\ref{formula2}) we obtain
\begin{equation}
\label{formula3}
\frac{\partial}{\partial a}\left(e^{\textstyle\beta\,a^+a}\right)=-[a^+,e^{\textstyle\beta\,a^+a}]=
e^{\textstyle\beta\,a^+a}\,a^+-a^+\,e^{\textstyle\beta\,a^+a}=(e^\beta-1)\,a^+e^{\textstyle\beta\,a^+a}\,,
\end{equation}
\begin{equation}
\label{formula4}
\frac{\partial}{\partial a^+}\left(e^{\textstyle\beta\,a^+a}\right)=[a,e^{\textstyle\beta\,a^+a}]=
a\,e^{\textstyle\beta\,a^+a}-e^{\textstyle\beta\,a^+a}\,a=(e^\beta-1)\,e^{\textstyle\beta\,a^+a}\,a\,.
\end{equation}
The following theorem first time appears probably in~\cite{20} without proof. We could not find any proofs of it. Therefore we give here our original proof.
\begin{theorem}
\label{factorization_U}
The unitary operator \(U\) can be presented in the form
$$
U=e^{\textstyle\lambda\,(a^2-a^{+2})}=e^{-\textstyle\gamma\,a^{+2}}\,e^{{\textstyle\beta\,(a^{+}a+\frac{1}{2})}}\,e^{\textstyle\gamma\,a^{2}}\,,
$$
where
$$
\displaystyle
2\gamma=\mbox{\rm th}(2\lambda)\,,\quad e^{\textstyle\beta}=\frac{1}{\mbox{\rm ch}(2\lambda)}\,.
$$
\end{theorem}
\begin{proof}
Let's present the operator \(U\) in the form
$$
U=e^{\textstyle\lambda\,(a^2-a^{+2})}=e^{-\textstyle\gamma\,a^{+2}}\,e^{\textstyle X}\,e^{\textstyle\gamma\,a^{2}}\,,
$$
where \(X\) is unknown operator and \(\gamma\) is unknown function of \(\lambda\). Then
$$
e^{\textstyle X}=e^{\textstyle\gamma\,a^{+2}}\,U\,e^{-\textstyle\gamma\,a^{2}}\,.
$$
Using lemma~(\ref{differentiation_function_of_a_and_a+}) we have
$$
\frac{\partial (e^X)}{\partial a^+}=e^{\textstyle\gamma\,a^{+2}}\left(2\gamma\,a^+\,U+[a,U]\,\right)e^{\textstyle\gamma\,a^{2}}=e^{\textstyle\gamma\,a^{+2}}\left(2\gamma\,U\,U^+\,a^+\,U+a\,U-U\,a\,\right)\,e^{-\textstyle\gamma\,a^{2}}=
$$
$$
=e^{\textstyle\gamma\,a^{+2}}\left(2\gamma\,U\,U^+\,a^+\,U+U\,U^+\,a\,U-U\,a\,\right)\,e^{-\textstyle\gamma\,a^{2}}=
$$
$$
=e^{\textstyle\gamma\,a^{+2}}\,U\,\left(2\gamma\,U^+\,a^+\,U+U^+\,a\,U-a\,\right)\,e^{-\textstyle\gamma\,a^{2}}\,.
$$
Using~(\ref{transformation_a_by_U}) and adjoint relation to it we have  
$$
\frac{\partial (e^X)}{\partial a^+}=e^{\textstyle\gamma\,a^{+2}}\,U\,\left(2\gamma\,(\mbox{ch}(2\lambda)\,a^+-\mbox{sh}(2\lambda)\,a)+(\mbox{ch}(2\lambda)\,a-\mbox{sh}(2\lambda)\,a^+)-a\,\right)\,e^{-\textstyle\gamma\,a^{2}}=
$$
$$
=e^{\textstyle\gamma\,a^{+2}}\,U\,\left(\,(2\gamma\,\mbox{ch}(2\lambda)-\mbox{sh}(2\lambda))\,a^+   +(-2\gamma\,\mbox{sh}(2\lambda)+\mbox{ch}(2\lambda)-1)\,a\,\right)\,e^{-\textstyle\gamma\,a^{2}}\,.
$$
If we put
$$
2\gamma=\mbox{th}(2\lambda)\,,
$$
then the previous expression takes the form
$$
\frac{\partial (e^X)}{\partial a^+}=\left(\frac{1}{\mbox{ch}(2\lambda)}-1\right)\,e^X\,a\,.
$$
By the same way one can deduce
$$
\frac{\partial (e^X)}{\partial a}=\left(\frac{1}{\mbox{ch}(2\lambda)}-1\right)\,a^+\,e^X\,.
$$
Compare these two expressions with (\ref{formula3}),(\ref{formula4}) one conclude that
$$
X=\beta\,a^+a+\alpha\,,\quad e^{\textstyle\beta}=\frac{1}{\mbox{ch}(2\lambda)}\,,
$$
where \(\alpha=\alpha(\lambda)\) is unknown function of \(\lambda\). It remains to find this function
\(\alpha(\lambda)\). We have found that
$$
U=e^{-\textstyle\gamma\,a^{+2}}\,e^{{\textstyle\beta\,a^{+}a}\,\,+\textstyle{\alpha}}\,e^{\textstyle\gamma\,a^{2}}\,.
$$
If \(|0\rangle\) is a normalized vector corresponding to the vacuum state in Dirac's notations (\(a\,|0\rangle=0\), \(\langle0|\,a^+=0\), \(\langle0|0\rangle=1\)), then from this formula it follows that
$$
\langle0|\,U\,|0\rangle=e^{\textstyle\alpha}\equiv\delta(\lambda)
$$
or
$$
\delta(\lambda)=\langle0|\,e^{\textstyle\lambda\,(a^2-a^{+2})}\,|0\rangle\,.
$$
One has next
$$
\delta'(\lambda)=\langle0|(a^2-a^{+2})\,e^{\textstyle\lambda\,(a^2-a^{+2})}\,|0\rangle=
\langle0|\,a^2\,U\,|0\rangle\,,
$$
$$
\delta'(\lambda)=\langle0|\,e^{\textstyle\lambda\,(a^2-a^{+2})}\,(a^2-a^{+2})\,|0\rangle=
-\langle0|\,U\,a^{+2}\,|0\rangle
$$
and hence
$$
2\,\delta'(\lambda)=\langle0|\,a^2\,U-U\,a^{+2}\,|0\rangle=\langle0|\,U\,U^+\,a^2\,U-U\,a^{+2}\,U^+\,U\,|0\rangle=
$$
$$
=\langle0|\,a^2\,U-U\,a^{+2}\,|0\rangle=\langle0|\,U\,(\mbox{ch}(2\lambda)\,a-\mbox{sh}(2\lambda)\,a^+)^2-(\mbox{ch}(2\lambda)\,a^++\mbox{sh}(2\lambda)\,a)^2\,U\,|0\rangle=
$$
$$
=\langle0|\,U\,(\mbox{sh}^2(2\lambda)\,a^{+2}-\mbox{sh}(2\lambda)\,\mbox{ch}(2\lambda)\,a\,a^+)-(\mbox{sh}^2(2\lambda)\,a^2+\mbox{sh}(2\lambda)\,\mbox{ch}(2\lambda)\,a\,a^+)\,U\,|0\rangle=
$$
$$
=\mbox{sh}^2(2\lambda)\,\langle0|\,U\,a^{+2}-a^2\,U\,|0\rangle-\mbox{sh}(2\lambda)\,\mbox{ch}(2\lambda)\,\langle0|\,U\,a\,a^{+}+a\,a^+\,U\,|0\rangle=
$$
$$
=-2\,\mbox{sh}^2(2\lambda)\,\delta'(\lambda)-\mbox{sh}(2\lambda)\,\mbox{ch}(2\lambda)\,\langle0|\,U+U\,|0\rangle=-2\,\mbox{sh}^2(2\lambda)\,\delta'(\lambda)-2\,\mbox{sh}(2\lambda)\,\mbox{ch}(2\lambda)\,\delta(\lambda)\,.
$$
Thus we obtain
$$
2\,\delta'(\lambda)=-2\,\mbox{sh}^2(2\lambda)\,\delta'(\lambda)-2\,\mbox{sh}(2\lambda)\,\mbox{ch}(2\lambda)\,\delta(\lambda)
$$
or
$$
\delta'(\lambda)=-\mbox{th}(2\lambda)\,\delta(\lambda)\,.
$$
Solving this differential equation with initial condition \(\delta(0)=1\) we obtain finally
$$
\delta(\lambda)=\frac{1}{\sqrt{\mbox{ch}(2\lambda)}}=e^{{\textstyle\beta/2}}\,\quad(\alpha=\beta/2)\,.
$$
\end{proof}

Using the factorization of the operator \(U\) which was established in the theorem~(\ref{factorization_U}), one can find the matrix representation \(U_{m,n}=(Ue_n,e_m)\) of this operator. Taking into account that \(a^+e_n=\sqrt{n+1}\,e_{n+1}\) and \(a\,e_n=\sqrt{n}\,e_{n-1}\) one has
$$
Ue_n=e^{-\textstyle\gamma\,a^{+2}}\,e^{{\textstyle\beta\,(a^{+}a+\frac{1}{2})}}\,e^{\textstyle\gamma\,a^{2}}\,e_n=e^{-\textstyle\gamma\,a^{+2}}\,e^{{\textstyle\beta\,(a^{+}a+\frac{1}{2})}}\,\sum\limits_{k=0}^\infty\frac{(\gamma\,a^2)^k}{k!}\,e_n=
$$
$$
=e^{-\textstyle\gamma\,a^{+2}}\,\sum\limits_{k=0}^{\textstyle[\,\frac{n}{2}\,]}\frac{\gamma^k}{k!}\,
e^{{\textstyle\beta\,(n-2k+\frac{1}{2})}}\sqrt{n(n-1)\ldots(n-2k+1)}\,e_{n-2k}=
$$
$$
=e^{{\textstyle\beta\,(n+\frac{1}{2})}}\,\sum\limits_{j=0}^\infty\frac{(-\gamma)^j}{j!}\sum\limits_{k=0}^{\textstyle[\,\frac{n}{2}\,]}\frac{\gamma^k}{k!}\,
e^{\textstyle-2\beta k}\sqrt{n(n-1)\ldots(n-2k+1)}\,\cdot
$$
$$
\cdot\sqrt{(n-2k+1)(n-2k+2)\ldots(n-2k+2j)}\,e_{n-2k+2j}\,.
$$
It follows that
$$
U_{m,n}=0\,,\quad m\not\equiv n\:\:(\mbox{mod}\:2)\,.
$$
and if \(m\equiv n\:\:(\mbox{mod}\:2)\) and \(m\ge n\) then (putting \(n-2k+2j=m\))
$$
U_{m,n}=e^{{\textstyle\beta\,(n+\frac{1}{2})}}\,(-\gamma)^{\textstyle\frac{m-n}{2}}\,\sum\limits_{k=0}^{\textstyle[\,\frac{n}{2}\,]}\frac{(-\gamma^2)^k\,e^{-2\beta k}}{(\frac{m-n}{2}+k)!\,k!}\,
\sqrt{n(n-1)\ldots(n-2k+1)}\,\cdot
$$
$$
\cdot\sqrt{(n-2k+1)(n-2k+2)\ldots m}=
$$
\begin{equation}
\label{U_mn_matrix}
=e^{{\beta\,(n+\frac{1}{2})}}\,(-\gamma)^{\textstyle\frac{m-n}{2}}\,\sqrt{n!\,m!}\,\sum\limits_{k=0}^{\textstyle[\,\frac{n}{2}\,]}(-1)^k\frac{\left(\gamma\,e^{-\beta}\right)^{2k}}{k!\,(n-2k)!\,(\frac{m-n}{2}+k)!}\,.
\end{equation}
If \(m\equiv n\:\:(\mbox{mod}\:2)\) and \(m\le n\) then
$$
U_{m,n}=e^{{\beta\,(n+\frac{1}{2})}}\,(-\gamma)^{\textstyle\frac{m-n}{2}}\,\sqrt{n!\,m!}\,\sum\limits_{k=\textstyle\frac{n-m}{2}}^{\textstyle[\,\frac{n}{2}\,]}(-1)^k\frac{\left(\gamma\,e^{-\beta}\right)^{2k}}{k!\,(n-2k)!\,(\frac{m-n}{2}+k)!}=
$$
$$
=e^{{\beta\,(m+\frac{1}{2})}}\,(-\gamma)^{\textstyle\frac{n-m}{2}}\,\sqrt{n!\,m!}\:(-1)^{\textstyle\frac{n-m}{2}}\sum\limits_{k=0}^{\textstyle[\,\frac{m}{2}\,]}(-1)^k\frac{\left(\gamma\,e^{-\beta}\right)^{2k}}{k!\,(m-2k)!\,(\frac{n-m}{2}+k)!}=(-1)^{\textstyle\frac{n-m}{2}}\,U_{n,m}\,.
$$
Let us apply the transformation \(U\) to the operator~(\ref{main_operator_two_Rabi_maodel}) :
$$
\tilde{H}=\tilde{H_0}+\tilde{V}\,,\quad \tilde{H}_0=U^+\,H_0\,U\,,\quad \tilde{V}=U^+\,V\,U\,.
$$
Here \(\tilde{H}_0\) is diagonal according~(\ref{diagonal_H_0}): \(\tilde{H}_0=\mbox{diag}\{E_n^{(0)}\}\). The matrix of the operator \(V\) is diagonal. Consider for example the matrix of the operator \(\hat{\bf V}_1\). Then \(V=\mbox{diag}\{V_{kk}\}\) and
$$
V_{kk}=\frac{\Delta}{2}\,(-1)^l\,,\quad k=2l\:\:\mbox{or}\:\:k=2l+1\,,\quad l\in\mathbb{Z}_+\,.
$$

\begin{lemma}
\label{U_V_permutation}
The operators \(U\) and \(V\) satisfy the condition
$$
U^+\,V=V\,U\quad\mbox{or}\quad U\,V\,U=V\,.
$$
\end{lemma}
\begin{proof}
Using the matrix representation of \(U\) and \(V\) one has for \(m\equiv n\:\:(\mbox{mod}\:2)\)
$$
(U^+\,V)_{n,m}=\sum_k U^+_{n,k}V_{k,m}=U^+_{n,m}V_{m,m}=U_{m,n}V_{m,m}=(-1)^{\textstyle\frac{n-m}{2}}\,U_{n,m}\,V_{m,m}\,.
$$
If \(n=2p\), \(m=2s\) then
$$
(U^+\,V)_{n,m}=(-1)^{p-s}\,U_{n,m}\,(-1)^s\,\frac{\Delta}{2}=U_{n,m}\,(-1)^p\,\frac{\Delta}{2}=
V_{n,n}\,U_{n,m}=(V\,U)_{n,m}\,.
$$
If \(n=2p+1\), \(m=2s+1\) we obtain the same equality. If \(m\not\equiv n\:\:(\mbox{mod}\:2)\) we have also
$$
(U^+\,V)_{n,m}=U_{m,n}V_{m,m}=0=V_{n,n}\,U_{n,m}=(V\,U)_{n,m}\,.
$$
\end{proof}
Note that the same situation as in lemma~(\ref{U_V_permutation}) takes place also in one-photon quantum Rabi model and first found in~\cite{11}.

Using this lemma we find
$$
\tilde{V}=U^+\,V\,U=V\,U\,U=V\,U(\lambda)\,U(\lambda)=V\,U(2\lambda)\,.
$$
Here we used a group property \(U(\lambda_1)\,U(\lambda_2)=U(\lambda_1+\lambda_2)\) of the orthogonal operator \(U\). Thus the matrix of \(\tilde{V}\) has the form
$$
\tilde{V}_{m,n}=V_{m,m}\,U_{m,n}(2\lambda)=(-1)^l\,\frac{\Delta}{2}\,U_{m,n}(2\lambda)\,,\quad
l=\left\{\begin{array}{r}
m/2\,,\quad m-\mbox{even}\\
(m-1)/2\,,\quad m-\mbox{odd}
\end{array} \right.\,.
$$
From the theorem~(\ref{factorization_U}) it follows that
$$
\gamma(\lambda)=\frac{\mbox{\rm th}(2\lambda)}{2}\,,\quad e^{\textstyle\beta(\lambda)}=\frac{1}{\mbox{\rm ch}(2\lambda)}\,.
$$
Hence using~(\ref{equation_for_lambda}) and (\ref{diagonal_H_0}) we find
$$
\gamma(2\lambda)=\frac{\mbox{\rm th}(4\lambda)}{2}=g\,,\quad e^{\textstyle\beta(2\lambda)}=\frac{1}{\mbox{\rm ch}(4\lambda)}=\sqrt{1-4g^2}=\omega\,.
$$
Thus from (\ref{U_mn_matrix}) we obtain
$$
U_{m,n}(2\lambda)=\omega^{{(n+\frac{1}{2})}}\,(-g)^{\textstyle\frac{m-n}{2}}\,\sqrt{n!\,m!}\,\sum\limits_{k=0}^{\textstyle[\,\frac{n}{2}\,]}(-1)^k\frac{\left(g/\omega\right)^{2k}}{k!\,(n-2k)!\,(\frac{m-n}{2}+k)!}=
$$
$$
=(-1)^{\textstyle\frac{m-n}{2}}\,\sqrt{\omega}\,g^{\textstyle\frac{m+n}{2}}\,\sqrt{n!\,m!}\,\sum\limits_{k=0}^{\textstyle[\,\frac{n}{2}\,]}(-1)^k\frac{\left(\omega/g\right)^{n-2k}}{k!\,(n-2k)!\,(\frac{m-n}{2}+k)!}\quad
(m\equiv n\:\:(\mbox{mod}\:2)\,,\:m\ge n)\,.
$$
Introduce the polynomials \(P_n^{(s)}(x)\) of degree \(n\) by the formula
\begin{equation}
\label{polynomials_P_n}
P_n^{(s)}(x)\equiv\sum\limits_{k=0}^{\textstyle[\,\frac{n}{2}\,]}(-1)^k\frac{n!\,(2x)^{n-2k}}{k!\,(n-2k)!\,(s+k)!}
\end{equation}
Through these polynomials the matrices \(U_{m,n}(2\lambda)\) and \(\tilde{V}_{m,n}\) can be presented as follows
$$
U_{m,n}(2\lambda)=(-1)^{\textstyle\frac{m-n}{2}}\,\sqrt{\omega}\,g^{\textstyle\frac{m+n}{2}}\,\sqrt{\frac{m!}{n!}}\,P_n^{(s)}\left(\frac{\omega}{2g}\right)\,,\quad s=\frac{m-n}{2}\,,
$$
$$
\tilde{V}_{m,n}=(-1)^{l+\textstyle\frac{m-n}{2}}\,\frac{\Delta}{2}\,\sqrt{\omega}\,g^{\textstyle\frac{m+n}{2}}\,\sqrt{\frac{m!}{n!}}\,P_n^{(s)}\left(\frac{\omega}{2g}\right)\,,\quad
l=\left\{\begin{array}{r}
m/2\,,\quad m-\mbox{even}\\
(m-1)/2\,,\quad m-\mbox{odd}
\end{array} \right.
$$
or
\begin{equation}
\label{matrix_of_V_tilde}
\tilde{V}_{m,n}=(-1)^{l}\,\frac{\Delta}{2}\,\sqrt{\omega}\,g^{\textstyle\frac{m+n}{2}}\,\sqrt{\frac{m!}{n!}}\,P_n^{(s)}\left(\frac{\omega}{2g}\right)\,,\quad
l=\left\{\begin{array}{r}
n/2\,,\quad n-\mbox{even}\\
(n-1)/2\,,\quad n-\mbox{odd}
\end{array} \right.\,.
\end{equation}
These formulas are valid if \(m\equiv n\:\:(\mbox{mod}\:2)\) and \(m\ge n\).

If \(m\equiv n\:\:(\mbox{mod}\:2)\) and \(m\le n\) then
$$
U_{m,n}(2\lambda)=(-1)^{\textstyle\frac{n-m}{2}}\,U_{n,m}(2\lambda)
$$
and
$$
\tilde{V}_{m,n}=(-1)^l\,\frac{\Delta}{2}\,\sqrt{\omega}\,g^{\textstyle\frac{m+n}{2}}\,\sqrt{\frac{n!}{m!}}\,P_m^{(s)}\left(\frac{\omega}{2g}\right)\,,\:\: s=\frac{n-m}{2}\,,\:\:
l=\left\{\begin{array}{r}
m/2\,,\quad m-\mbox{even}\\
(m-1)/2\,,\quad m-\mbox{odd}
\end{array} \right.
$$
that is
$$
\tilde{V}_{m,n}=\tilde{V}_{n,m}
$$
as it should be as the matrix \(V\) is symmetric and the transformation \(U\) is orthogonal. If \(m\not\equiv n\:\:(\mbox{mod}\:2)\) then
$$
\tilde{V}_{m,n}=U_{m,n}=0\,.
$$
It follows that the vectors \(e_n\) with even and odd indexes form invariant subspaces for the operators \(H\) and \(\tilde{H}\).

\section{Asymptotics of polynomials \(P_n^{(s)}(x)\)}

The polynomials \(P_n^{(s)}(x)\) are defined by formula~(\ref{polynomials_P_n}):
$$
P_n^{(s)}(x)=\sum\limits_{k=0}^{\textstyle[\,\frac{n}{2}\,]}(-1)^k\frac{n!\,(2x)^{n-2k}}{k!\,(n-2k)!\,(s+k)!}\,,\quad s=\frac{m-n}{2}\ge 0\,.
$$
We are interesting in the behaviour of this polynomials for large \(n\) and \(m\). We will follow the Liouville-Steklov method~\cite{14,15}.

Consider the hypergeometric function \(F(a,b;c;x)\)
$$
F(a,b;c;x)=1+\sum\limits_{k=1}^\infty\frac{a(a+1)\ldots(a+k-1)\,b(b+1)\ldots(b+k-1)}{c(c+1)\ldots(c+k-1)}\,\frac{x^k}{k!}\,.
$$
This function satisfies the differential equation
$$
x(1-x)\,\frac{d^2F}{dx^2}+[c-(a+b+1)x\,]\,\frac{dF}{dx}-ab\,F=0\,.
$$

I. Consider first the case of even \(n\) and \(m\). Changing \(n\to 2n\), \(m\to 2m\), consider the polynomials
$$
P_{2n}^{(s)}(x)=\sum\limits_{k=0}^{n}(-1)^k\frac{(2n)!\,(2x)^{2n-2k}}{k!\,(2n-2k)!\,(s+k)!}\,,\quad s=m-n\ge 0\,.
$$
Consider the hypergeometric function \(F(-n,-m;1/2;-x^2)\). One has
$$
F(-n,-m;1/2;-x^2)=\sum_{k=0}^n(-1)^k\frac{n!\,m!}{(n-k)!\,(m-k)!}\,\frac{(2x)^{2k}}{(2k)!}=
$$
$$
=(-1)^n\,\frac{n!\,m!}{(2n)!}\sum_{k=0}^n(-1)^k\frac{(2n)!\,(2x)^{2n-2k}}{k!\,(m-n+k)!\,(2n-2k)!}=
(-1)^n\,\frac{n!\,m!}{(2n)!}\,P_{2n}^{(m-n)}(x)\,.
$$
It follows that
$$
P_{2n}^{(m-n)}(x)=(-1)^n\,\frac{(2n)!}{n!\,m!}\,F(-n,-m;1/2;-x^2)\,.
$$
Denoting \(z=-x^2\), one has
$$
z(1-z)\,\frac{d^2F}{dz^2}+[1/2-(1-n-m)z\,]\,\frac{dF}{dz}-n\,m\,F=0\,,
$$
whence
$$
(1+x^2)\,\frac{d^2F}{dx^2}+(1-2n-2m)\,x\,\frac{dF}{dx}+4\,n\,m\,F=0\,.
$$
Making the substitution \(x=\mbox{sh}\,t\) we obtain
$$
F''_{tt}-(2n+2m)\,\mbox{th}\,t\,F'_t+4\,m\,n\,F=0\,.
$$
After standard substitution \(F(-n,-m;1/2;-\,\mbox{sh}^2\,t)\equiv F(t)=s(t)\,u(t)\), where 
$$
s(t)=\exp\left\{(n+m)\int\mbox{th}\,t\,dt\right\}=\left(\mbox{ch}\,t\right)^{n+m}\,,
$$
we come to the following equation for the function \(u(t)\)
$$
u''(t)+\left[4nm-(n+m)^2\,\mbox{th}^2t+\frac{n+m}{\mbox{ch}^2t}\right]u(t)=0
$$
or
$$
u''(t)+Q(t)\,u(t)=0\,,
$$
$$
Q(t)=4nm-(n+m)^2\,\mbox{th}^2t+\frac{n+m}{\mbox{ch}^2t}=
\frac{(n+m)^2+n+m}{\mbox{ch}^2t}-(m-n)^2\,.
$$
If \(m,n\to\infty\) so that \(\dfrac{m-n}{m+n}\to0\), then \(Q(t)> 0\) and \(Q(t)\to+\infty\) for any fixed \(t\). So to find the asymptotics of \(u(t)\) for large \(m\) and \(n\) we can use the WKB method ideas. But it is better to use the Liouville transformation~\cite{16}:
\begin{equation}
\label{Liouville_transformation}
y=\int\limits_0^t\sqrt{Q(\tau)}\,d\tau\,,\quad u=Q^{-1/4}\,z\,.
\end{equation}
This transformation reduces differential equation \(u''+Q(t)\,u=0\) to the following equation for the function \(z(y)\):
\begin{equation}
\label{diff_equation_for_z(y)}
z''+(1+q(y))\,z=0\,,\quad q=Q^{-3/4}\,\frac{d^2}{dt^2}\left(Q^{-1/4}\right)=-\frac{Q''_{tt}}{4\,Q^2}+\frac{5\,Q'^2_t}{16\,Q^3}\,.
\end{equation}
From the formula for \(q(y)\) it follows that \(q(y)\to 0\) as \(n,m\to\infty\). So we can expect that \(z(y)\) tends to the solution of oscillator equation \(z''+z=0\). To prove it consider the differential equation for \(z(y)\) as inhomogeneous one
$$
z''(y)+z=f(u)\,,\quad f(u)=-q(y)\,z\,.
$$ 
Using the known formulas we obtain
$$
z(y)=z'(0)\sin y+z(0)\cos y+\int\limits_0^y f(v)\,\sin(y-v)\,dv
$$
The initial conditions \(z(0)\), \(z'(0)\) can be found from the initial conditions for the function \(F(x)\equiv F(-n,-m;1/2;-\,x^2)\): \(F(0)=1\), \(F'(0)=0\). The same conditions we obtain and for the function \(f(t)\equiv F(\mbox{sh}\,t)=s(t)\,u(t)\): \(f(0)=1\), \(f'(0)=0\). Next the formula
$$
f(t)=s(t)\,u(t)=\left(\mbox{ch}\,t\right)^{n+m}\,u(t)
$$
yields \(u(0)=1\), \(u'(0)=0\). Now the Liouville transformation gives \(y(0)=0\) and
$$
z(0)=\left(Q(0)\right)^{1/4}\,u(0)=\sqrt[4]{4nm+n+m}\,.
$$
From
$$
u'(t)=-\frac{Q'(t)}{4Q^{5/4}(t)}\,z(y)+Q^{-1/4}(t)\,z'(y)\,y'(t)=
-\frac{Q'(t)}{4Q^{5/4}(t)}\,z(y)+Q^{1/4}(t)\,z'(y)
$$
and \(Q'(0)=0\) it follows that \(z'(0)=0\).

Thus we obtain
$$
z(y)=\sqrt[4]{4nm+n+m}\,\cos y-\int\limits_0^y q(v)\,z(v)\,\sin(y-v)\,dv\,.
$$
To estimate the integral in this expression one can construct successive iterations of this integral equation (successive approximations). Putting \(k=\sqrt[4]{4nm+n+m}\) and substituting the integral equation in itself one has
$$
z(y)=k\cos y-k\int\limits_0^y q(v)\,\cos v\,\sin(y-v)\,dv+\int\limits_0^y q(v)\,\sin(y-v)\int\limits_0^v q(v_1)\,z(v_1)\,\sin(v-v_1)\,dv_1\,dv\,.
$$
Continuing this process we obtain
$$
z(y)=k\cos y+k\sum\limits_{n=1}^\infty (-1)^n\,z_n(y)\,,
$$
where
$$
z_n(y)=\int\limits_0^y q(y_1)\sin(y-y_1)\,dy_1\ldots
$$
$$
\int\limits_0^{y_{n-2}} q(y_{n-1})\sin(y_{n-2}-y_{n-1})\,dy_{n-1}\int\limits_0^{y_{n-1}} q(y_n)\cos y_n\sin(y_{n-1}-y_n)\,dy_n\,.
$$
Standard logic~\cite{17} leads to the estimate
$$
|z_n(y)|\le\int\limits_0^y|q(y_1)|\,dy_1\ldots\int\limits_0^{y_{n-2}} |q(y_{n-1})|\,dy_{n-1}\int\limits_0^{y_{n-1}} |q(y_n)|\,dy_n=\frac{1}{n!}\left[\int\limits_0^y|q(y)|\,dy\right]^n\,.
$$
Hence
\begin{equation}
\label{estimation_z(y)}
|z(y)-k\cos y|\le k\left(\exp\left\{\int\limits_0^y|q(y)|\,dy\right\}-1\right)\,.
\end{equation}
Denote \(\lambda=\sqrt{(n+m)^2+n+m}\). Then
$$
Q(t)=\frac{\lambda^2}{\mbox{ch}^2t}-s^2
$$
and
$$
y=\int\limits_0^t\sqrt{Q(\tau)}\,d\tau=\int\limits_0^t\sqrt{\frac{\lambda^2}{\mbox{ch}^2\tau}-s^2}\,d\tau\sim\lambda\,\arctan(\mbox{sh}\,t)\,,\quad n\to\infty\,.
$$
In particular, if \(n=m\) then \(s=0\) and we have exactly \(y=\lambda\,\arctan(\mbox{sh}\,t)\). As \(t\) is fixed and \(y\) grows with \(n\) and \(m\) we should change variable in the integral in~(\ref{estimation_z(y)}):
$$
\int\limits_0^y|q(y)|\,dy=\int\limits_0^t|q(y(\tau))|\,y'(\tau)\,d\tau=\int\limits_0^t|q(y(\tau))|\,\sqrt{Q(\tau)}\,d\tau=\int\limits_0^t\left|-\frac{Q''_{tt}(\tau)}{4\,Q^2(\tau)}+\frac{5\,Q'^2_t(\tau)}{16\,Q^3(\tau)}\right|\,\sqrt{Q(\tau)}\,d\tau\,.
$$
From this formula and above formula for \(Q(t)\) it follows that if \(\dfrac{s}{\lambda}\to 0\) as \(n\to\infty\) then 
$$
\int\limits_0^y|q(y)|\,dy=O\left(\frac{1}{\lambda}\right)\,,\quad n\to\infty\,.
$$
(As \(m\ge n\), \(m\to\infty\) too.) From estimation~(\ref{estimation_z(y)}) it follows that
$$
z(y)=k\left(\cos y+O\left(\frac{1}{\lambda}\right)\right)\,.
$$
Coming back to the function \(F(t)=s(t)\,u(t)\) we obtain
$$
F(t)=\left(\mbox{ch}\,t\right)^{n+m}\,Q^{-1/4}(t)\,z(y(t))
=\left(\mbox{ch}\,t\right)^{n+m}\,\left(\frac{\lambda^2}{\mbox{ch}^2t}-s^2\right)^{-1/4}\,k\left(\cos y+O\left(\frac{1}{\lambda}\right)\right)=
$$
$$
=\left(\mbox{ch}\,t\right)^{n+m}\,\left(\frac{1}{\mbox{ch}^2t}-\frac{s^2}{\lambda^2}\right)^{-1/4}\,\frac{k}{\sqrt{\lambda}}\left(\cos y+O\left(\frac{1}{\lambda}\right)\right)\,,\quad n\to\infty\,,
$$
where
$$
y=\lambda\int\limits_0^t\sqrt{\frac{1}{\mbox{ch}^2\tau}-\frac{s^2}{\lambda^2}}\,d\tau\,.
$$ 
Because
$$
\frac{k}{\sqrt{\lambda}}=\frac{\sqrt[4]{4nm+n+m}}{\sqrt{\lambda}}=\frac{(\lambda^2-s^2)^{1/4}}{\sqrt{\lambda}}=\left(1-\frac{s^2}{\lambda^2}\right)^{1/4} \quad\mbox{and}\quad O\left(\frac{1}{\lambda}\right)=O\left(\frac{1}{n+m}\right)
$$
we find finally
$$
F(t)=\left(\mbox{ch}\,t\right)^{n+m}\,\left(\frac{1}{\mbox{ch}^2t}-\frac{s^2}{\lambda^2}\right)^{-1/4}\left(1-\frac{s^2}{\lambda^2}\right)^{1/4}\left(\cos y+O\left(\frac{1}{n+m}\right)\right)\,.
$$
Coming back to the variable \(x=\mbox{sh}\,t\) we obtain the following result:
\begin{theorem}
\label{asymptot_Polynom_even}
Let \(\lambda=\sqrt{(n+m)^2+n+m}\) and \(s=m-n\ge 0\). If \(\dfrac{s}{\lambda}\to0\) as \(n\to\infty\) then
$$
P_{2n}^{(m-n)}(x)=(-1)^n\,\frac{(2n)!}{n!\,m!}\,(1+x^2)^{\textstyle\frac{n+m}{2}}\left(\frac{1}{1+x^2}-\frac{s^2}{\lambda^2}\right)^{-1/4}\left(1-\frac{s^2}{\lambda^2}\right)^{1/4}\left(\cos y+O\left(\frac{1}{n+m}\right)\right)
$$
as \(n\to\infty\), where 
$$
y=\lambda\int\limits_0^{\mbox{\small arsh\,x}}\sqrt{\frac{1}{\mbox{ch}^2\tau}-\frac{s^2}{\lambda^2}}\,d\tau\,.
$$
The estimation of the reminders is uniform on any bounded set of \(x\).
\end{theorem}

II. Consider the case of odd \(n\) and \(m\). Changing \(n\to 2n+1\), \(m\to 2m+1\), consider the polynomials
$$
P_{2n+1}^{(s)}(x)=\sum\limits_{k=0}^{n}(-1)^k\frac{(2n+1)!\,(2x)^{2n+1-2k}}{k!\,(2n+1-2k)!\,(s+k)!}\,,\quad s=m-n\ge 0\,.
$$
These polynomials are related to the hypergeometric function \(F(-n,-m;3/2;-x^2)\). Actually, we have
$$
F(-n,-m;3/2;-x^2)=\sum_{k=0}^n(-1)^k\frac{n!\,m!}{(n-k)!\,(m-k)!}\,\frac{(2x)^{2k}}{(2k+1)!}=
$$
$$
=(-1)^n\,\frac{n!\,m!}{(2n+1)!}\sum_{k=0}^n(-1)^k\frac{(2n+1)!\,(2x)^{2n-2k}}{k!\,(m-n+k)!\,(2n+1-2k)!}=(-1)^n\,\frac{n!\,m!}{(2n+1)!}\,\frac{P_{2n+1}^{(m-n)}(x)}{2x}\,.
$$
Hence
$$
P_{2n+1}^{(m-n)}(x)=(-1)^n\,\frac{(2n+1)!}{n!\,m!}\,2x\,F(-n,-m;3/2;-x^2)\,.
$$
The hypergeometric function \(F(-n,-m;3/2;z)\) satisfies the differential equation
$$
z(1-z)\,\frac{d^2F}{dz^2}+[3/2-(1-n-m)z\,]\,\frac{dF}{dz}-n\,m\,F=0\,,
$$
The substitution \(z=-x^2\) leads to the equation
$$
(1+x^2)\,\frac{d^2F}{dx^2}+\left(\frac{2}{x}+(1-2n-2m)\,x\right)\frac{dF}{dx}+4\,n\,m\,F=0\,.
$$
For the function \(f(x)=x\,F(-n,-m;3/2;-x^2)\) we obtain the equation
$$
(1+x^2)\,f''-(1+2n+2m)\,x\,f'+(2n+1)(2m+1)\,f=0\,.
$$
Making the substitution \(x=\mbox{sh}\,t\) as in the case of even indexes we obtain
$$
f''_{tt}-2(1+n+m)\,\mbox{th}\,t\,f'_t+(2n+1)(2m+1)\,f=0\,.
$$
To remove the term with \(f'\) make the substitution \(f(t)=s(t)\,u(t)\), where 
$$
s(t)=\exp\left\{(1+n+m)\int\mbox{th}\,t\,dt\right\}=\left(\mbox{ch}\,t\right)^{n+m+1}\,.
$$
Then we come to the following equation for the function \(u(t)\)
$$
u''(t)+\left[(2n+1)(2m+1)-(1+n+m)^2\,\mbox{th}^2t+\frac{1+n+m}{\mbox{ch}^2t}\right]u(t)=0
$$
or
$$
u''(t)+Q(t)\,u(t)=0\,,
$$
$$
Q(t)=\frac{(n+m+1)^2+n+m+1}{\mbox{ch}^2t}-(m-n)^2\,.
$$
As before if \(m,n\to\infty\) so that \(\dfrac{m-n}{m+n}\to0\), then \(Q(t)\to+\infty\) for any fixed \(t\). The Liouville transformation~(\ref{Liouville_transformation}) gives for the function \(z(y)\) the differential equation~(\ref{diff_equation_for_z(y)}). Considering this equation as before we obtain
$$
z(y)=z'(0)\sin y+z(0)\cos y-\int\limits_0^y q(v)\,z(v)\,\sin(y-v)\,dv\,.
$$
But the initial conditions here defer from the even indexes case. For the function \(f(x)\) we have \(f(0)=0\), \(f'(0)=1\). The same conditions we obtain and for the function \(f(t)\). For the function \(u(t)\) we obtain \(u(0)=0\), \(u'(0)=1\). Now the Liouville transformation gives \(y(0)=0\) and
$$
z(0)=\left(Q(0)\right)^{1/4}\,u(0)=0\,.
$$
From the formula
$$
u'(t)=-\frac{Q'(t)}{4Q^{5/4}(t)}\,z(y)+Q^{-1/4}(t)\,z'(y)\,y'(t)=
-\frac{Q'(t)}{4Q^{5/4}(t)}\,z(y)+Q^{1/4}(t)\,z'(y)
$$
it follows that
$$
z'(0)=Q^{-1/4}(0)=((2n+1)(2m+1)+n+m+1)^{-1/4}\equiv k\,.
$$
Thus we obtain
$$
z(y)=k\,\sin y-\int\limits_0^y q(v)\,z(v)\,\sin(y-v)\,dv\,.
$$
Application successive iterations to this integral equation gives the estimation like~(\ref{estimation_z(y)}):
$$
|z(y)-k\sin y|\le k\left(\exp\left\{\int\limits_0^y|q(y)|\,dy\right\}-1\right)\,.
$$
Denote \(\lambda=\sqrt{(n+m+1)^2+n+m+1}\). Then
$$
Q(t)=\frac{\lambda^2}{\mbox{ch}^2t}-s^2
$$
and if \(\dfrac{s}{\lambda}\to 0\) as \(n\to\infty\) then as in the previous case we obtain
$$
\int\limits_0^y|q(y)|\,dy=O\left(\frac{1}{\lambda}\right)\,,\quad n\to\infty
$$
so that
$$
z(y)=k\left(\sin y+O\left(\frac{1}{\lambda}\right)\right)\,.
$$
Because \((2n+1)(2m+1)+n+m+1=\lambda^2-s^2\) we have
$$
k=\left(\lambda^2-s^2\right)^{-1/4}=\frac{1}{\sqrt{\lambda}}\left(1-\frac{s^2}{\lambda^2}\right)^{-1/4}\,.
$$
Coming back to the function \(f(t)=s(t)\,u(t)\) we obtain
$$
f(t)=\left(\mbox{ch}\,t\right)^{n+m+1}\,Q^{-1/4}(t)\,z(y(t))=
$$
$$
=\left(\mbox{ch}\,t\right)^{n+m+1}\,\left(\frac{\lambda^2}{\mbox{ch}^2t}-s^2\right)^{-1/4}\frac{1}{\sqrt{\lambda}}\left(1-\frac{s^2}{\lambda^2}\right)^{-1/4}\left(\sin y+O\left(\frac{1}{\lambda}\right)\right)=
$$
$$
=\left(\mbox{ch}\,t\right)^{n+m+1}\,\left(\frac{1}{\mbox{ch}^2t}-\frac{s^2}{\lambda^2}\right)^{-1/4}\left(1-\frac{s^2}{\lambda^2}\right)^{-1/4}\frac{1}{\lambda}\left(\sin y+O\left(\frac{1}{\lambda}\right)\right)\,,\quad n\to\infty\,,
$$
where
$$
y=\lambda\int\limits_0^t\sqrt{\frac{1}{\mbox{ch}^2\tau}-\frac{s^2}{\lambda^2}}\,d\tau\,.
$$ 
Returning to the variable \(x=\mbox{sh}\,t\) we obtain the following result:
\begin{theorem}
\label{asymptot_Polynom_odd}
Let \(\lambda=\sqrt{(n+m+1)^2+n+m+1}\) and \(s=m-n\ge 0\). If \(\dfrac{s}{\lambda}\to0\) as \(n\to\infty\) then
$$
P_{2n+1}^{(m-n)}(x)=(-1)^n\,\frac{(2n+1)!}{n!\,m!}\,(1+x^2)^{\textstyle\frac{n+m+1}{2}}\left(\frac{1}{1+x^2}-\frac{s^2}{\lambda^2}\right)^{-1/4}\left(1-\frac{s^2}{\lambda^2}\right)^{-1/4}\cdot
$$
$$
\cdot\,\frac{2}{\lambda}\left(\sin y+O\left(\frac{1}{n+m}\right)\right)
$$
as \(n\to\infty\), where 
$$
y=\lambda\int\limits_0^{\mbox{\small arsh\,x}}\sqrt{\frac{1}{\mbox{ch}^2\tau}-\frac{s^2}{\lambda^2}}\,d\tau\,.
$$
The estimation of the reminders is uniform on any bounded set of \(x\).
\end{theorem}

\section{Eigenvalues asymptotics of two-photon quantum Rabi model}

Using theorems~(\ref{asymptot_Polynom_even}) and (\ref{asymptot_Polynom_odd}) we can find an important asimptotics and estimations of the perturbation matrices~(\ref{matrix_of_V_tilde}) and which are required in the theorem~(\ref{noncompact_perturbation_oscillator}).  
 
First note that the spectrum of the operator \(\tilde{H}_0\) (formula~(\ref{diagonal_H_0}) ) satisfies the conditions of the theorem~(\ref{noncompact_perturbation_oscillator}) for the operator \(D\). The eigenvectors \(\{e_n\}\) formed an orthonormal basis in \(H\) and the eigenvalues \(E_n^{(0)}=\omega\,(n+1/2)-1/2\) have the oscillator type. The constant shift \(-1/2\) and the factor \(\omega\) are not important for the application of the theorem~(\ref{noncompact_perturbation_oscillator}) because they can be easily reduced to the condition in theorem by constant shift and dividing on \(\omega\). Hence without loss of generality we can regard for now that \(\omega=1\) and \(E_n^{(0)}=n\).
 
Show that the operator \(\tilde{V}\) satisfies the all conditions on the operator \(R\) in the theorem~(\ref{noncompact_perturbation_oscillator}). The perturbation \(\tilde{V}\) is bounded symmetric and non-compact because 
$$
\tilde{V}^2=\frac{\Delta^2}{4}\,E\,.
$$
From the formula~(\ref{matrix_of_V_tilde}) and the theorem~(\ref{asymptot_Polynom_even}) it follows that
$$
|\tilde{V}_{2m,2n}|=\frac{\Delta}{2}\,\sqrt{\omega}\,g^{m+n}\,\sqrt{\frac{(2m)!}{(2n)!}}\,\left|P_{2n}^{(m-n)}\left(\frac{\omega}{2g}\right)\right|\le C\,\sqrt{\frac{(2n)!}{n!^2}\,\frac{(2m)!}{m!^2}}\,2^{-(n+m)}
$$
for all sufficiently large \(n\) and \(m\) if \(\dfrac{m-n}{n+m}\to 0\), where \(C\) is constant independent of \(n\) and \(m\). Here we used that 
$$
\sqrt{1+x^2}=\sqrt{1+\frac{\omega^2}{4g^2}}=\frac{1}{2g}\,.
$$
Using the formula \(C_{2n}^n=\dfrac{(2n)!}{n!^2}=\dfrac{2^{2n}}{\sqrt{\pi n}}\,(1+O(1/n))\) as \(n\to\infty\) we obtain
\begin{equation}
\label{estimation_V_even}
|\tilde{V}_{2m,2n}|\le\frac{C_1}{(nm)^{1/4}}\,.
\end{equation}
Analogously from the formula~(\ref{matrix_of_V_tilde}) and the theorem~(\ref{asymptot_Polynom_odd}) we have
$$
|\tilde{V}_{2m+1,2n+1}|=\frac{\Delta}{2}\,\sqrt{\omega}\,g^{m+n+1}\,\sqrt{\frac{(2m+1)!}{(2n+1)!}}\,\left|P_{2n+1}^{(m-n)}\left(\frac{\omega}{2g}\right)\right|\le C_2\,\sqrt{\frac{(2n+1)!}{n!^2}\,\frac{(2m+1)!}{m!^2}}\,\frac{2^{-(n+m)}}{n+m}\,.
$$
for all sufficiently large \(n\) and \(m\) if \(\dfrac{m-n}{n+m}\to 0\), where \(C_2\) is constant independent of \(n\) and \(m\). Acting as before we obtain
\begin{equation}
\label{estimation_V_odd}
|\tilde{V}_{2m+1,2n+1}|\le C_2\,\sqrt{\frac{(2n)!}{n!^2}\,\frac{(2m)!}{m!^2}}\,\frac{2^{-(n+m)}\sqrt{(2n+1)(2m+1)}}{n+m}\le\frac{C_3}{(nm)^{1/4}}\,.
\end{equation}
Combining (\ref{estimation_V_even}) and (\ref{estimation_V_odd}) we have
\begin{equation}
\label{estimation_V}
|\tilde{V}_{m,n}|\le\frac{C}{(nm)^{1/4}}
\end{equation}
for all sufficiently large \(n\) and \(m\) if \(\dfrac{m-n}{n+m}\to0\) as \(n\to\infty\). In particular, if \(m-n=p=\mbox{const}\) this formula gives at once
$$
\lim\limits_{n\to\infty}\tilde{V}_{n+p,n}=\lim\limits_{n\to\infty}\tilde{V}_{n,n+p}=0\quad \forall p\in\mathbb{Z}\,.
$$
Denoting the eigenvalues of \(\tilde{H}=\tilde{H}_0+\tilde{V}\) by \(\lambda_n\) and applying theorem~(\ref{noncompact_perturbation_oscillator}) one has 
\begin{equation}
\label{lambda_n_Rabi}
\lambda_n=n+\tilde{V}_{nn}+\sum\limits_{k\ne n}\frac{\tilde{V}_{nk}^2}{n-k}+O\left((\Delta_n+\|Ke_n\|)\,(|\tilde{V}_{nn}|+\|Ke_n\|)\right)\,,\quad n\to\infty
\end{equation}
where
$$
\|Ke_n\|^2=\sum_{k\ne n}\frac{\tilde{V}_{kn}^2}{(k-n)^2}\,,\quad\Delta_n^2=O\left(\sum\limits_{m=1}^{\infty}\frac{(|\tilde{V}_{mm}|+\|Ke_m\|)^2}{\|(\tilde{H}_0-nE-E_n/2)e_m\|^2}\right)\,.
$$
The definition of \(E_n\) is in the theorem~(\ref{eigenvalues_at_relatively_compact_R}).
Let us estimate the different terms in this expression. 

Using~(\ref{estimation_V}) we have
$$
\|Ke_n\|^2=\sum_{k\ne n\atop k\ge0}\frac{\tilde{V}_{kn}^2}{(k-n)^2}=\sum_{p\ne 0\atop p\ge -n}\frac{\tilde{V}_{n+p,n}^2}{p^2}=\sum_{0<|p|\le\sqrt{n}}\frac{\tilde{V}_{n+p,n}^2}{p^2}+\sum_{|p|>\sqrt{n}\atop p\ge -n}\frac{\tilde{V}_{n+p,n}^2}{p^2}\le
$$
$$
\le C^2\sum_{0<|p|\le\sqrt{n}}\frac{1}{p^2\sqrt{n(n+p)}}+\frac{1}{n}\sum_{|p|>\sqrt{n}\atop p\ge -n}\tilde{V}_{n+p,n}^2\le\frac{C^2}{n}\sum_{p\ne 0}\frac{1}{p^2}+\frac{1}{n}\sum_{k\ge 0}\tilde{V}_{k,n}^2=
\frac{\pi C^2}{3n}+\frac{\Delta^2}{4n}\,.
$$ 
Here we used the formula \(\tilde{V}^2=\dfrac{\Delta^2}{4}\,E\). Hence
$$
\|Ke_n\|=O\left(\frac{1}{\sqrt{n}}\right)\,.
$$
By the same way one can estimate the sum
$$
\left|\sum\limits_{k\ne n\atop k\ge 0}\frac{\tilde{V}_{nk}^2}{n-k}\right|\le\sum\limits_{k\ne n\atop k\ge 0}\frac{\tilde{V}_{nk}^2}{|n-k|}=\sum_{p\ne 0\atop p\ge -n}\frac{\tilde{V}_{n+p,n}^2}{|p|}=\sum_{0<|p|\le\textstyle\frac{n}{\ln n}}\frac{\tilde{V}_{n+p,n}^2}{|p|}+\sum_{|p|>\textstyle\frac{n}{\ln  n}\atop p\ge -n}\frac{\tilde{V}_{n+p,n}^2}{|p|}\le
$$
$$
\le C^2\sum_{0<|p|\le\textstyle\frac{n}{\ln  n}}\frac{1}{|p|\sqrt{n(n+p)}}+\frac{\ln n}{n}\sum_{|p|>\textstyle\frac{n}{\ln  n}\atop p\ge -n}\tilde{V}_{n+p,n}^2\le\frac{C^2}{\sqrt{n}}\sum_{0<|p|\le\textstyle\frac{n}{\ln  n}}\frac{1}{|p|\sqrt{n+p}}+\frac{\ln n}{n}\sum_{k\ge 0}\tilde{V}_{k,n}^2=
$$
$$
=\frac{C^2}{\sqrt{n}}\,O\left(\frac{\ln n}{\sqrt{n}}\right)+\frac{\ln n}{n}\,\frac{\Delta^2}{4}\,.
$$ 
It follows that
$$
\sum\limits_{k\ne n\atop k\ge 0}\frac{\tilde{V}_{nk}^2}{n-k}=O\left(\frac{\ln n}{n}\right)\,.
$$

Using again the theorems~(\ref{asymptot_Polynom_even}) and~(\ref{asymptot_Polynom_odd}) we find the following asymptotics for the diagonal elements \(\tilde{V}_{nn}\):
$$
\tilde{V}_{2n,2n}=\frac{\Delta}{2}\,\sqrt{\omega}\,\frac{(2n)!}{n!^2}\,2^{-2n}\,\frac{1}{\sqrt{2g}}\left(\cos\left(A\sqrt{4n^2+2n}\right)+O\left(\frac{1}{n}\right)\right)=
$$
$$
=\frac{\Delta}{2}\,\sqrt{\frac{\omega}{2\pi g n}}\,\cos\left(A\sqrt{4n^2+2n}\right)+O\left(\frac{1}{n^{3/2}}\right)\,,
$$
$$
\tilde{V}_{2n+1,2n+1}=\frac{\Delta}{2}\,\sqrt{\omega}\,\frac{(2n+1)!}{n!^2}\,2^{-2n-1}\,\frac{1}{\sqrt{2g}}\,\frac{2}{\sqrt{(2n+1)^2+2n+1}}\left(\sin\left(A\sqrt{(2n+1)^2+2n+1}\right)+\right.
$$
$$
\left.+O\left(\frac{1}{n}\right)\right)=
\frac{\Delta}{2}\,\sqrt{\frac{\omega}{\pi g(2n+1)}}\,\sin\left(A\sqrt{(2n+1)^2+2n+1}\right)+O\left(\frac{1}{n^{3/2}}\right)\,,
$$
where \(A=\arctan x=\arctan\dfrac{\omega}{2g}=\arctan\dfrac{\sqrt{1-4g^2}}{2g}\).
These formulas can be rewrite as follows
$$
\tilde{V}_{nn}=(-1)^{\textstyle[\,\frac{n}{2}\,]}\,\frac{\Delta}{2}\,\sqrt{\frac{\sqrt{1-4g^2}}{\pi g\,n}}\,\cos\left(A\sqrt{n^2+n}-\frac{\pi}{2}\,n\right)+O\left(\frac{1}{n^{3/2}}\right)\,,\quad n\to\infty
$$
or
$$
\tilde{V}_{nn}=(-1)^{\textstyle[\,\frac{n}{2}\,]}\,\frac{\Delta}{2}\,\sqrt{\frac{\sqrt{1-4g^2}}{\pi g\,n}}\,\cos\left(A\,(n+1/2)-\frac{\pi}{2}\,n\right)+O\left(\frac{1}{n^{3/2}}\right)\,,\quad n\to\infty\,,
$$
$$
A=\arctan\dfrac{\sqrt{1-4g^2}}{2g}\,.
$$

Using the results above we can estimate \(\Delta_n^2\):
$$
\Delta_n^2=O\left(\sum\limits_{m=1}^{\infty}\frac{1}{m\,\|(\tilde{H}_0-nE-E_n/2)e_m\|^2}\right)
$$
One has
$$
\sum\limits_{m=1}^{\infty}\frac{1}{m\,\|(\tilde{H}_0-nE-E_n/2)e_m\|^2}=\sum\limits_{m=1}^{n-1}\frac{1}{m\,(n-m-1/2)^2}+\frac{4}{n}+\sum\limits_{m=n+1}^\infty\frac{1}{m\,(m-n-1/2)^2}=
$$
$$
=\sum\limits_{m=1}^{n-1}\frac{1}{(n-m)\,(m-1/2)^2}+\frac{4}{n}+\sum\limits_{m=1}^\infty\frac{1}{(m+n)\,(m-1/2)^2}\le
$$
$$
\le\sum\limits_{m=1}^{n-1}\frac{1}{(n-m)\,(m-1/2)^2}+\frac{4}{n}+\frac{1}{n}\,\sum\limits_{m=1}^\infty\frac{1}{(m-1/2)^2}\,.
$$
Decompose the fraction in first sum on partial fractions one conclude that it equals also to \(O(1/n)\). Thus
$$
\Delta_n=O\left(\frac{1}{\sqrt n}\right)\,.
$$
Combining all estimates of this section, from~(\ref{lambda_n_Rabi}) we obtain
$$
\lambda_n=n+(-1)^{\textstyle[\,\frac{n}{2}\,]}\,\frac{\Delta}{2}\,\sqrt{\frac{\sqrt{1-4g^2}}{\pi g\,n}}\,\cos\left(A\,(n+1/2)-\frac{\pi}{2}\,n\right)+O\left(\frac{\ln n}{n}\right)\,,\quad n\to\infty\,.
$$
Returning to the eigenvalues \(E_n^{(0)}=\omega(n+1/2)-1/2\) of the operator \(\tilde{H}_0\) we can write
$$
\lambda_n=n\,\sqrt{1-4g^2}+\frac{\sqrt{1-4g^2}-1}{2}+(-1)^{\textstyle[\,\frac{n}{2}\,]}\,\frac{\Delta}{2}\,\sqrt{\frac{\sqrt{1-4g^2}}{\pi g\,n}}\,\cos\left(A\,(n+1/2)-\frac{\pi}{2}\,n\right)+O\left(\frac{\ln n}{n}\right)
$$
\(n\to\infty\). This is three-term asymptotic formula. It defines the eigenvalues asymptotics of the operator \(\hat{\bf H}_1\) from the section~(\ref{Rabi}), corresponding to the perturbation 
\(\hat{\bf V}_1\). For the operator \(\hat{\bf H}_2\), corresponding to the perturbation 
\(\hat{\bf V}_2=-\hat{\bf V}_1\), we will have the same asymptotic formula but with the sign "$-$" before \(\tilde{V}_{nn}\).

Denoting the eigenvalues of the operators \(\hat{\bf H}_1\) and \(\hat{\bf H}_2\) by \(E^{\pm}_n\) respectively, we obtain the following result

\begin{theorem}
\label{final_theorem_asyptot}
If \(0<g<1/2\) then the eigenvalues \(E^{\pm}_n\) of the quantum Rabi model Hamiltonian \(\hat{\bf H}\) ~{\rm(\ref{full_hamiltonian_Rabi})} have the following three-term asymptotic formula
$$
E^{\pm}_n=n\,\sqrt{1-4g^2}+\frac{\sqrt{1-4g^2}-1}{2}\pm(-1)^{\textstyle[\,\frac{n}{2}\,]}\,\frac{\Delta}{2}\,\sqrt{\frac{\sqrt{1-4g^2}}{\pi g\,n}}\,\cos\left(A\,(n+1/2)-\frac{\pi}{2}\,n\right)+O\left(\frac{\ln n}{n}\right)
$$
\(n\to\infty\), where \(A=\arctan\dfrac{\sqrt{1-4g^2}}{2g}\).
\end{theorem}

Note that two-term asymptotic formula was obtained in~\cite{3}. From this formula it follows that the eigenvalues have oscillating behaviour as it takes place in one-photon quantum Rabi model~\cite{11,21,22,23}.

\section{Conclusion}

In the previous section we could see that all corrections to the unperturbed eigenvalues \(E_n^{(0)}\) have order \(O(1/n)\) except \(\lambda_n^{(1)}=\tilde{V}_{nn}\) and 
$$
\lambda_n^{(2)}=\sum\limits_{k\ne n\atop k\ge 0}\frac{\tilde{V}_{nk}^2}{n-k}\,.
$$
\(\lambda_n^{(1)}\) gives us the third term in asymptotics but \(\lambda_n^{(2)}=O(\ln n/n)\). Wherein we estimated \(\lambda_n^{(2)}\) by modulus without taking into account the signs of terms in sum. If we will take into account the signs the estimation of \(\lambda_n^{(2)}\) can be also equal
\(O(1/n)\). Therefore we want to end by the following hypothesis: 

{\bf Hypothesis:} In the theorem (\ref{final_theorem_asyptot}) one can change the remainder \(O(\ln n/n)\) by \(O(1/n)\).

\end{document}